\documentclass[12pt]{amsart}

\usepackage{amsmath}
\usepackage{amsxtra}
\usepackage{amscd}
\usepackage{amsthm}
\usepackage{amsfonts}
\usepackage{amssymb}
\usepackage{mathrsfs}
\usepackage{eucal}
\usepackage{color}
\usepackage[all]{xy}
\usepackage{enumerate}
\usepackage[colorlinks=true, pdfstartview=FitV, linkcolor=blue, citecolor=blue, urlcolor=blue]{hyperref}

\newcommand{\nc}{\newcommand}
\newcommand{\rc}{\renewcommand}

\numberwithin{equation}{section}

\theoremstyle{plain}
\newtheorem{thm}{Theorem} [section]
\newtheorem{prop}[thm]{Proposition}
\newtheorem{lem}[thm]{Lemma}

\newtheorem{conj}[thm]{Conjecture}
\newtheorem{assumption}[thm]{Assumption}

\newtheorem{important note}[thm]{Important Note}
\newtheorem{problem}[thm]{Problem}

\theoremstyle{definition}
\newtheorem{rmk}[thm]{Remark}

\nc{\on}{\operatorname}

\nc{\Lemma}{\begin{lem}}
\nc{\enlemma}{\end{lem}}
\nc{\Proof}{\begin{proof}}
\nc{\QED}{\end{proof}}
\nc{\Prop}{\begin{prop}}
\nc{\enprop}{\end{prop}}

\nc{\bA}{{\mathbb A}}
\nc{\bB}{{\mathbb B}}
\nc{\bC}{{\mathbb C}}
\nc{\bD}{{\mathbb D}}
\nc{\bE}{{\mathbb E}}
\nc{\bF}{{\mathbb F}}
\nc{\bG}{{\mathbb G}}
\nc{\bH}{{\mathbb H}}
\nc{\bI}{{\mathbb I}}
\nc{\bJ}{{\mathbb J}}
\nc{\bK}{{\mathbb K}}
\nc{\bL}{{\mathbb L}}
\nc{\bM}{{\mathbb M}}
\nc{\bN}{{\mathbb N}}
\nc{\bO}{{\mathbb O}}
\nc{\bP}{{\mathbb P}}
\nc{\bQ}{{\mathbb Q}}
\nc{\bR}{{\mathbb R}}
\nc{\bS}{{\mathbb S}}
\nc{\bT}{{\mathbb T}}
\nc{\bU}{{\mathbb U}}
\nc{\bV}{{\mathbb V}}
\nc{\bW}{{\mathbb W}}
\nc{\bZ}{{\mathbb Z}}
\nc{\bX}{{\mathbb X}}
\nc{\bY}{{\mathbb Y}}

\newcommand{\bbZ}{\mathbb{Z}}

\nc{\cA}{{\mathcal A}}
\nc{\cB}{{\mathcal B}}
\nc{\cC}{{\mathcal C}}
\nc{\cD}{{\mathcal D}}
\nc{\cE}{{\mathcal E}}
\nc{\cF}{{\mathcal F}}
\nc{\cG}{{\mathcal G}}
\nc{\cH}{{\mathcal H}}
\nc{\cI}{{\mathcal I}}
\nc{\cJ}{{\mathcal J}}
\nc{\cK}{{\mathcal K}}
\nc{\cL}{{\mathcal L}}
\nc{\cM}{{\mathcal M}}
\nc{\cN}{{\mathcal N}}
\nc{\cO}{{\mathcal O}}
\nc{\cP}{{\mathcal P}}
\nc{\cQ}{{\mathcal Q}}
\nc{\cR}{{\mathcal R}}
\nc{\cS}{{\mathcal S}}
\nc{\cT}{{\mathcal T}}
\nc{\cU}{{\mathcal U}}
\nc{\cV}{{\mathcal V}}
\nc{\cW}{{\mathcal W}}
\nc{\cZ}{{\mathcal Z}}
\nc{\cX}{{\mathcal X}}
\nc{\cY}{{\mathcal Y}}

\nc{\fA}{{\mathfrak A}}
\nc{\fB}{{\mathfrak B}}
\nc{\fC}{{\mathfrak C}}
\nc{\fD}{{\mathfrak D}}
\nc{\fE}{{\mathfrak E}}
\nc{\fF}{{\mathfrak F}}
\nc{\fG}{{\mathfrak G}}
\nc{\fH}{{\mathfrak H}}
\nc{\fI}{{\mathfrak I}}
\nc{\fJ}{{\mathfrak J}}
\nc{\fK}{{\mathfrak K}}
\nc{\fL}{{\mathfrak L}}
\nc{\fM}{{\mathfrak M}}
\nc{\fN}{{\mathfrak N}}
\nc{\fO}{{\mathfrak O}}
\nc{\fP}{{\mathfrak P}}
\nc{\fQ}{{\mathfrak Q}}
\nc{\fR}{{\mathfrak R}}
\nc{\fS}{{\mathfrak S}}
\nc{\fT}{{\mathfrak T}}
\nc{\fU}{{\mathfrak U}}
\nc{\fV}{{\mathfrak V}}
\nc{\fW}{{\mathfrak W}}
\nc{\fZ}{{\mathfrak Z}}
\nc{\fX}{{\mathfrak X}}
\nc{\fY}{{\mathfrak Y}}
\nc{\fa}{{\mathfrak a}}
\nc{\fb}{{\mathfrak b}}
\nc{\fc}{{\mathfrak c}}
\nc{\fd}{{\mathfrak d}}
\nc{\fe}{{\mathfrak e}}
\nc{\ff}{{\mathfrak f}}
\nc{\fg}{{\mathfrak g}}
\nc{\fh}{{\mathfrak h}}
\nc{\fiI}{{\mathfrak i}}  
\nc{\ffi}{{\mathfrak i}}  
\nc{\fj}{{\mathfrak j}}
\nc{\fk}{{\mathfrak k}}
\nc{\fl}{{\mathfrak{l}}}
\nc{\fm}{{\mathfrak m}}
\nc{\fn}{{\mathfrak n}}
\nc{\fo}{{\mathfrak o}}
\nc{\fp}{{\mathfrak p}}
\nc{\fq}{{\mathfrak q}}
\nc{\fr}{{\mathfrak r}}
\nc{\fs}{{\mathfrak s}}
\nc{\ft}{{\mathfrak t}}
\nc{\fu}{{\mathfrak u}}
\nc{\fv}{{\mathfrak v}}
\nc{\fw}{{\mathfrak w}}
\nc{\fz}{{\mathfrak z}}
\nc{\fx}{{\mathfrak x}}
\nc{\fy}{{\mathfrak y}}

\nc{\al}{{\alpha }}
\nc{\ga}{{\gamma }}
\nc{\de}{{\delta }}
\nc{\del}{{\partial }}
\nc{\ep}{{\varepsilon }}
\nc{\vap}{{\epsilon }}

\nc{\ze}{{\zeta }}
\nc{\et}{{\eta }}
\rc{\th}{{\theta }}

\nc{\vth}{{\vartheta }}

\nc{\io}{{\iota }}
\nc{\ka}{{\kappa }}
\nc{\la}{{\lambda }}
\nc{\vrho}{{\varrho}}
\nc{\si}{{\sigma }}
\nc{\ups}{{\upsilon }}
\nc{\vphi}{{\varphi }}
\nc{\om}{{\omega }}

\nc{\Ga}{{\Gamma }}
\nc{\De}{{\Delta }}
\nc{\nab}{{\nabla}}
\nc{\Th}{{\Theta }}
\nc{\La}{{\Lambda }}
\nc{\Si}{{\Sigma }}
\nc{\Ups}{{\Upsilon }}
\nc{\Om}{{\Omega }}

\newcommand{\A}{{\mathcal A}}

\newcommand{\M}{{\mathcal M}}

\nc{\hE}{\widehat{\cE}}
\nc{\hA}{\widehat{A}}
\nc{\hK}{\widehat{K}}
\nc{\hM}{\widehat{\cM}}
\nc{\hN}{\widehat{\cN}}
\nc{\hL}{\widehat{\cL}}
\nc{\hF}{\hat{\cF}}
\nc{\hcA}{\widehat\cA}
\nc{\hcK}{\widehat\cK}
\nc{\tN}{\widetilde\cN}
\nc{\tM}{\widetilde\cM}

\nc{\Coh}{{{\mathcal C}oh}}
\nc{\Loc}{{{\mathcal L}oc}}
\nc{\GR}{{G_\bR}}

\newcommand{\Ext}{{\operatorname {Ext}}}
\newcommand{\ct}{{T^*X}}

\newcommand{\imbed}{\hookrightarrow}

\newcommand{\To}[1][{\quad}]{\xrightarrow{\;#1\;}}

 \providecommand{\curlybrackets}[1]{{\{\!\{#1\}\!\}}}
 \nc{\cb}{\curlybrackets}

\nc{\str}{{\hcA}}

\nc{\Spec}{{\on{Spec}}}
\nc{\cext}{{\mathscr{E}\mspace{-2mu}xt}}
\nc{\ctor}{{\cT\mathit o\mathit r}}
\nc{\crhom}{{\operatorname{R}\cH\mathit o\mathit m}}
\newcommand{\chom}{\mathscr{H}\mspace{-4mu}om}
\nc{\oh}{{\on{H}}}
\nc{\codim}{{\on{codim}}}
\nc{\Supp}{{\on{Supp}}}
\nc{\coker}{{\on{Coker}}}
\renewcommand{\ker}{\Ker}
\nc{\id}{\mathrm{id}}
\nc{\seteq}{\mathbin{:=}}
\newcommand{\isoto}[1][]{\xrightarrow[#1]%
{{\raisebox{-.6ex}[0ex][-.6ex]{$\mspace{1mu}\sim\mspace{2mu}$}}}}

\nc{\dA}[1]{\operatorname{D_\cA}(#1)}
\nc{\cl}{\colon}
\nc{\Ker}{\on{Ker}}

\nc{\rhD}{{\operatorname{Mod}_{hr}(\cD_X)}}
\nc{\rhDL}{{\operatorname{Mod}_{hr}(\cD_X)_\La}}
\nc{\rhE}{{\operatorname{Mod}_{hr}(\cE_X)}}
\nc{\rhEL}{{\operatorname{Mod}_{hr}(\cE_X)_\La}}
\nc{\rhFE}{{\operatorname{Mod}_{hr}(\hat\cE_X)}}
\nc{\rhFEL}{{\operatorname{Mod}_{hr}(\hat\cE_X)_\La}}
\nc{\op}{{\on{P}}}
\nc{\oM}{{\on{M}}}
\nc{\ba}{\begin{array}}
\nc{\ea}{\end{array}}
\nc{\hs}{\hspace*}

\newcommand{\set}[2]{\left\{#1 \mathbin{;} #2 \right\}}
\newcommand{\scbul}{{\,\raise.4ex\hbox{$\scriptscriptstyle\bullet$}\,}}
\nc{\dT}{{\mathring{T}}{}^*}
\nc{\oY}{{\mathring{Y}}}
\nc{\oLa}{{\mathring{\La}}}
\nc{\eq}{\begin{eqnarray}}
\nc{\eneq}{\end{eqnarray}}
\nc{\eqn}{\begin{eqnarray*}}
\nc{\eneqn}{\end{eqnarray*}}
\nc{\Per}{\on{\mathcal{P}\kern-.25ex\mathit{er}}}

\nc{\bigmid}{\;\mathbin{\rule[-1.8ex]{.5pt}{3.8ex}}\;}

\nc{\od}{{\operatorname{D}}}
\nc{\odh}{{\operatorname{D}_{\hcA}}}
\nc{\gl}{{\mathfrak{gl}}}

\newcommand{\indlim}[1][]{\mathop{\varinjlim}\limits_{#1}}

\nc{\Z}{\bbZ}
\nc{\C}{\bC}

\nc{\be}{\begin{enumerate}}
\nc{\ee}{\end{enumerate}}
\nc{\bnum}{\be[{\rm(i)}]}
\nc{\enum}{\ee}
\nc{\Mod}{\on{Mod}}
\nc{\coh}{\mathrm{coh}}
\nc{\sing}{\mathrm{sing}}
\nc{\drho}{{\mathring{\rho}}}
\nc{\dU}{{\mathring{U}}}
\nc{\ssum}{\mathop{\mbox{\normalsize$\sum$}}\limits}
\nc{\bbL}{\mathbb{L}at}
\nc{\bl}{\bigl(}
\nc{\br}{\bigr)}

\nc{\ro}{{\rm(}}
\nc{\rf}{{\rm)}}
\nc{\DA}{{\cD\cA}}
\nc{\DK}{{\cD\cK}}
\nc{\hDA}{\widehat{\cD\cA}}
\nc{\hDK}{{\widehat{\cD\cK}}}
\nc{\gd}{{\mathrm{good}}}
\nc{\rhofin}{\rho-\mathrm{fin}}
\nc{\bna}{\be[{\rm(a)}]}
\newcommand{\proolim}[1][]{\mathop{``{\varprojlim}''}\limits_{#1}}
\nc{\epito}{{\twoheadrightarrow}}
\nc{\monoto}{{\rightarrowtail}}
\nc{\prolim}{\varprojlim}
\nc{\bwr}{{\mbox{\large$\wr$}}}
\nc{\cors}{\mathbf{k}}

\begin{document}

\title{Microdifferential systems and the codimension-three conjecture}

\author{Masaki Kashiwara}

\address{Research Institute for Mathematical Sciences,
Kyoto University,
Kyoto, 606--8502, Japan\\
and \\ Department of Mathematical Sciences, Seoul National University,
Seoul, Korea}
\thanks{M.\ K.\ was partially supported
by Grant-in-Aid for Scientific Research (B) 23340005,
Japan Society for the Promotion of Science.}
\email{masaki@kurims.kyoto-u.ac.jp}

\author{Kari Vilonen}

\address{Department of Mathematics, Northwestern University, Evanston, IL
60208, USA \\
and \\ Department of Mathematics, Helsinki University, Helsinki, Finland}
\thanks{K.\ V.\  was supported by NSF and by DARPA via AFOSR grant FA9550-08-1-0315}
\email {vilonen@math.northwestern.edu, vilonen@math.helsinki.fi}



\date{May 30, 2013}

\maketitle

\tableofcontents

\section{Introduction}

In this paper we give a proof of a fundamental conjecture,
{\em the codimension-three conjecture},
for microdifferential holonomic systems with regular singularities.
This conjecture emerged at the end of the 1970's and is well-known among experts. As far as we know, it was never formally written down as a conjecture, perhaps because of lack of concrete evidence for it. Our result can also be interpreted from a topological point of view as a statement about  microlocal perverse sheaves. However, our proof is entirely in the context of microdifferential holonomic systems.

Let $X$ be a complex manifold. We write $\cD_X$ for the sheaf of linear differential operators on $X$ with holomorphic coefficients. The Riemann-Hilbert correspondence identifies the categories of perverse sheaves and regular holonomic $\cD_X$-modules (\cite{K0}). Both of these notions are now widely used in mathematics. The study of microdifferential systems, $\cE_X$-modules, was initiated in \cite{SKK} where the basic properties were proved and some structural results were obtained. In the next section we recall the definition of the sheaf $\cE_X$ along with its basic properties. The notion of micro-support of sheaves was introduced in \cite{KS1, KS2}. Making use of this notion allows one to study perverse sheaves microlocally, i.e., locally on the cotangent bundle. In addition, using this notion, one can define microlocal perverse sheaves and establish the microlocal Riemann-Hilbert correspondence between regular holonomic $\cE_X$-modules and microlocal perverse sheaves. See, for example, \cite{A1,A2, W}.

Let us fix a conic Lagrangian subvariety $\La\subset \ct$.
It is often important and interesting
to understand the category $\op_\La(X)$ of perverse sheaves on $X$ with complex coefficients
whose micro-support  lies in $\La$.
Equivalently, thinking in terms of $\cD_X$-modules
we can view $\op_\La(X)$
as the category of regular holonomic  $\cD_X$-modules
whose characteristic variety is contained in $\La$. The basic structure of this category has been studied by several authors, for example,
\cite{Be}, \cite{K1}, \cite{KK}, \cite{MV} and \cite{SKK}.
In \cite{GMV1} it is shown
how in principle one can describe this category: when $X$ is algebraic
it is equivalent to the category of finitely generated modules
over a finitely presented  associative algebra.
However, it is perhaps more interesting to describe $\op_\La(X)$
in terms of the geometry of $\ct$. This is the the point of view we adopt here.

The category $\op_\La(X)$ gives rise to a stack $\Per_\La$ on $\ct$.
As we explained above, we can view this stack as either
the stack of microlocal perverse sheaves with support on $\La$
or as  the stack of regular holonomic $\cE_X$-modules supported on $\La$.
As $\cD_X$ is a subsheaf of rings of $\cE_X$,
the passage from $\op_\La(X)$ to $\Per_\La$ in this language is rather simple
as it amounts to merely extending the coefficients from $\cD_X$ to $\cE_X$.
One expects the microlocal description of $\op_\La(X)$,
i.e., the description of $\Per_\La(\ct)$ to make things conceptually simpler.
Our resolution of codimension-three-conjecture is a key step in this direction.

Let us write
\begin{equation}
\La  \ = \ \La^0\sqcup\La^1 \sqcup \La^ 2 \sqcup \cdots\,,
\end{equation}
where $\La^i$ is the locus of codimension $i$ singularities of $\La$. The appropriate notion of singularity in our context amounts to a Whitney stratification of $\La$.
We set $\La^{\geq i}=\cup_{k\ge i}\La^k$.
It is not difficult  to show,
either from the topological (\cite[Proposition 10.3.10]{KS2}) or
from the analytic point of view (\cite[Theorem 1.2.2]{KK}), that  the following two statements hold for $U\subset\ct$ an open subset:
\begin{equation}\label{codim1}
\text{The functor $\Per_\La(U)\To\Per_\La(U\setminus\La^{\geq 1})$
is faithful,}
\end{equation}
and
\begin{equation}\label{codim2}
\text{The functor $\Per_\La(U)\to\Per_\La(U\setminus\La^{\geq 2})$
is fully faithful.}
\end{equation}

In particular, the latter implies
\begin{equation}
\parbox{65ex}
{If we have a Lagrangian $\La =\La_1\cup\La_2$
with each $\La_i$ Lagrangian and
$\codim_\La(\La_1\cap\La_2 )\geq 2$, then  $\Per_\La(U) =
\Per_{\La_1}(U) \times  \Per_{\La_2}(U)$.}
\end{equation}
In concrete terms, \eqref{codim2} means that
beyond the codimension one singularities of $\La$
only conditions on objects are imposed.
All the essential data are already given along $\La^0$ and $\La^1$.
Along the locus $\La^0$ we specify a (twisted) local system (cf.\ \cite{K1}),
and along  $\La^1$  we specify some {``glue''}
between the local systems on various components of
$\La^0$. This ``glue'' may also impose conditions on the local system on $\La^0$.
Such a description of the stack $\Per_\La(U\setminus  \La^{\geq 2})$,
in terms of Picard-Lefschetz/Morse theory is given in \cite{GMV2}.

In this paper we answer the question as to what happens
beyond codimension two, i.e., we prove the following fundamental fact.

\begin{thm}
For an open subset $U$ of $\La$ and a closed analytic subset $Y$ of $U\cap\La$
of codimension {\em at least three in $\La$},
the functor $\Per_\La(U)\to\Per_\La(U\setminus Y)$ is
an equivalence of categories.
\end{thm}

Note that microlocal perverse sheaves can be defined with coefficients
in any field. As our methods are analytic, our theory applies only to the case when the field is of characteristic zero.
We do not know if our results are true beyond this case.

As was already noted, in our proof we work entirely
within the context of $\cE_X$-modules.
Our arguments apply to any holonomic module which possesses a (global)
 $\cE_X(0)$-lattice; here $\cE_X(0)$ stands for micro differential operators of order  at most zero. All regular holonomic $\cE_X$-modules do possess
such a lattice but we do not know if this is
true in the irregular case in general.

We will now formulate our results in more detail and explain the strategy of proof.
Our main extension theorem takes the following form:
\begin{thm}
\label{1}
Let $U$ be an open subset of $T^*X$, $\Lambda$ a closed Lagrangian
analytic subset of $U$, and $Y$ a closed analytic subset of $\Lambda$
of codimension {\em at least three}.
Let $\cM$ be a holonomic $\bl\cE_X\vert_{U\setminus Y}\br$-module
whose support is contained in $\Lambda\setminus Y$.
Assume that $\cM$ possesses an   $\bl\cE_X(0)\vert_{U\setminus Y}\br$-lattice.
Then $\cM$ extends uniquely to a holonomic module
defined on $U$ whose support is contained on $\La$.
\end{thm}
There is also the following version for submodules:
\begin{thm}
\label{2} Let $U$ be an open subset of $T^*X$, $\Lambda$ a closed Lagrangian
analytic subset of $U$, and $Y$ a closed analytic subset of $\Lambda$
of codimension {\em at least two}.
Let $\cM$ be a holonomic $\bl\cE_X\vert_U\br$-module whose support is contained in $\Lambda$ and let  $\cM_1$ be an
$\bl\cE_X\vert_{U\setminus Y}\br$-submodule of $\cM\vert_{U\setminus Y}$.
 Then  $\cM_1$  extends uniquely to a holonomic $\bl\cE_X\vert_U\br$-submodule of $\cM$.
\end{thm}

We deduce these results, which we call convergent versions, from their formal versions. In the formal versions we work over the the ring of formal microdifferential operators $\hE_X$ instead. Here are the statements in the formal case:

\begin{thm}
\label{3}
 Let $U$ be an open subset of $T^*X$, $\Lambda$ a closed Lagrangian
analytic subset of $U$, and $Y$ a closed analytic subset of $\Lambda$
of codimension {\em at least three}.
Let $\hM$ be a holonomic  $\bl\hE_X\vert_{U\setminus Y}\br$-module
whose support is contained in $\Lambda\setminus Y$.
Assume that $\cM$ possesses an $\bl\hE_X(0)\vert_{U\setminus Y}\br$-lattice.
Then $\hM$ extends uniquely to a holonomic module
defined on $U$ whose support is contained on $\La$.
\end{thm}
And:
\begin{thm}
\label{4}
Let $U$ be an open subset of $T^*X$, $\Lambda$ a closed Lagrangian
analytic subset of $U$, and $Y$ a closed analytic subset of $\Lambda$
of codimension {\em at least two}.
Let $\hM$ be a holonomic $\bl\hE_X\vert_U\br$-module whose support is contained in $\Lambda$ and let $\hM_1$ be an
 $\bl\hE_X\vert_{U\setminus Y}\br$-submodule of $\cM\vert_{U\setminus Y}$.
Then $\hM_1$ extends uniquely to a holonomic $\bl\hE_X\vert_U\br$-submodule of $\hM$.
\end{thm}
Let $j\cl U\setminus Y\to U$ be the open inclusion.
The theorems above amount to proving that the natural sheaf extensions
$j_*\cM$, $j_*\cM'$, $j_*\hM$ and $j_*\hM_1$  of $\cM$, $\cM_1$, $\hM$, and $\hM_1$, respectively, are coherent.
By a standard technique in several complex variables,
which in this context was already used in \cite{SKK},
it suffices to prove the coherence of a module after pushing it forward
under a map which is finite on the support of the module.
 Via this technique we are able to ``eliminate'' the extraneous variables and reduce extension problems, i.e., the question of coherence
of the sheaf extension, to a simpler form. The sheaf $\cE_X(0)$ can then be replaced by the sheaf $\cA_X$ whose precise definition is given in section 3. The $\cA_X$ is a commutative sheaf of rings and it can be viewed as a certain kind of neighborhood of $X\times\{0\}$ in $X\times \bC$. The sheaf $\cA_X$ has its formal version $\hcA_X$ and we can similarly replace $\hE_X(0)$ with $\hcA_X$ . The sheaf $\hcA_X$ is the structure sheaf of the formal neighborhood of $X\times\{0\}$ in $X\times \bC$.
We will call extension theorems
involving the sheaves $\cA_X$ and $\hcA_X$
commutative versions of the extension theorems.

As the proof of the submodule theorem is simpler and goes along the lines of the proof of the codimension three extension theorem we will focus on the codimension three extension theorem in the introduction and just briefly comment on the submodule case.

We will next state a theorem which by the discussion above implies the formal version of the codimension three extension theorem. Let $Y\subset X$ be a
subvariety and let us write $j\cl X\setminus Y\hookrightarrow X$
for the open embedding. Then:
\begin{thm}
\label{5}
If $\hN$ is a reflexive coherent $\str_{X\setminus Y}$-module and $\dim Y \leq \dim X -3$ then $j_*\hN$ is a coherent $\str_X$-module.
\end{thm}
The notion of {\em reflexive} is defined in the usual way. We call $\hN$ {\em reflexive} if $\hN\isoto\odh\odh{\hN}$, where
we have written $\odh{\hN}$ for the dual of $\hN$, i.e.,
\begin{equation*}
\odh{\hN} \ =  \ \chom_{\str_X}(\hN,\str_X)\,.
\end{equation*}

In section~\ref{The formal case} of this paper we deduce this result from the classical extension theorem due to Trautmann, Frisch-Guenot, and Siu \cite{T,FG,  Siu1} which we state here in the form suitable for us:
\begin{thm}[Trautmann, Frisch-Guenot, and Siu]
\label{6}
If $\cF$ is a reflexive coherent $\cO_{X\setminus Y}$-module and and $\dim Y \leq \dim X -3$
then $j_*\cF$ is a coherent $\cO_X$-module.
\end{thm}
This result is explicitly stated in this form in \cite[Theorem 5]{Siu1}. For a coherent $\cO_X$-module $\cF$,
we will write  $\cF^*= \chom_{\cO_{X}}(\cF, \cO_{X})$ for the dual of $\cF$ and we recall that $\cF$ is called reflexive if $\cF\isoto \cF^{**}$\,. In this paper $\cF^*$ will always mean the $\cO_X$-dual of $\cF$ even if $\cF$ carries some additional structure.

Similarly, the ``convergent'' case of the codimension three extension theorem would follow from:
\begin{conj}
\label{7}
If $\cN$ is a reflexive coherent $\cA_{X\setminus Y}$-module and $\dim Y \leq \dim X -3$, then $j_*\cN$ is a coherent $\cA_X$-module.
\end{conj}
We believe that this conjecture can be proved by extending the techniques of Trautmann, Frisch-Guenot, and Siu to our context. However, in this paper we proceed differently. Making use of the formal codimension three extension Theorem~\ref{3} allows us to make stronger assumptions on the $\cA_{X\setminus Y}$-module $\cN$. Namely, the codimension three extension Theorem~\ref{1}
follows from the formal codimension three extension theorem combined with:
\begin{thm}
\label{8}
Let X be a complex manifold and $Y$ a subvariety of $X$.
Let $\cN$ be a locally free $\cA_{X\setminus Y}$-module
  of finite rank
and let us assume further that the corresponding formal module
$\hcA_{X\setminus Y}\otimes_{\cA_{X\setminus Y}}\cN$ extends
to a locally free $\hcA_X$-module defined on all of $X$. If $\dim Y \leq \dim X -2$ then  $\cN$ also extends to a locally free $\cA_X$-module defined on $X$.
\end{thm}
We give a proof of this theorem in section~\ref{comparison}. In the proof we make use of a result of Bungart ~\cite[Theorem 8.1]{Bu}. He extends the Oka-Cartan principle to bundles whose structure group  $B^\times$
is the group of units of a Banach algebra $B$:
\begin{thm}[Bungart]
\label{intro statement}
On a Stein space there  is a natural bijection between isomorphism classes of holomorphic $B^\times$-bundles and topological $B^\times$-bundles.
\end{thm}
 Note that we are proving a slightly stronger statement than necessary for the codimension three conjecture as we have replaced the inequality $\dim Y \leq \dim X -3$ by $\dim Y \leq \dim X -2$. The stronger version is used in the proof of the submodule extension theorem which we now turn to briefly.

In the case of the submodule theorems it is more convenient  for us
to work with the equivalent quotient versions.
The formal version version of the submodule extension theorem follows from
\begin{thm}
\label{9}
Let $X$ be a complex manifold and $Y$ a closed submanifold
of codimension at least two and let $j\cl X\setminus Y\to X$ be the embedding.
Let $\hN$ be a coherent  $\hcA_X$-module,
$\hL$ a torsion free coherent $\hcA_{X\setminus Y}$-module
and let $\vphi\cl j^{-1}\hN\epito\hL$ be an epimorphism  of
 $\hcA_{X\setminus Y}$-modules.
Then the image of $\hN\to j_*\hL$ is a coherent $\hcA_X$-module.
\end{thm}
Just as in the case of the formal codimension three extension theorem we will make use of a classical result which is due to
Siu-Trautmann:

\begin{thm}[Siu and Trautmann]
\label{Siu Trautmann}
Let $\cF$ be a coherent $\cO_X$-module on $X$ and let $\cG$ be
is coherent $\cO_{X\setminus Y}$-module with a
homomorphism $j^*\cF \to\cG$ and we assume that codimension of $Y$ is at least two. If $\cG$ is torsion free, then $\operatorname{Im}(\cF \to j_*\cG)$ is coherent.
\end{thm}
This is a special case of \cite[Theorem 9.3]{ST}.

\medskip
The paper is organized as follows.

In section 2 of this paper
we recall  the notion of the ring of microdifferential operators $\cE_X$ and the notions of holonomic and
regular holonomic $\cE_X$-modules and prove the uniqueness part of the extension theorems.

In section 3 we introduce the sheaf of rings $\cA_X$ and its formal version $\hcA_X$ which will be crucial for our arguments. These are sheaves of commutative rings on $X$ and hence we call extension theorems involving these sheaves of rings on $X$ the ``commutative'' versions.

In section 4 we explain a general mechanism utilizing finite morphisms which allows us to pass between the extension problems for microdifferential operators and the extension problems in the commutative case. We do this both in the formal and convergent cases. A crucial ingredient in our arguments is the Quillen conjecture which was proved by Popescu, Bhatwadekar, and Rao \cite{P,BR}.

Section 5 contains the proofs of our main extension Theorems~\ref{3}, ~\ref{4}, ~\ref{1}, and ~\ref{2}. Each theorem is proved in a separate subsection. We make use of Theorems~\ref{5}, ~\ref{9}, and~\ref{8} which are proved in their own sections~\ref{The formal case}, ~\ref{The formal submodule theorem}, and ~\ref{comparison}, respectively.

In section~\ref{The formal case} we prove Theorem~\ref{5} by making use of the classical Theorem~\ref{6} of
Trautmann, Frisch-Guenot, and Siu \cite{T,FG,  Siu1}.

In section~\ref{The formal submodule theorem} we prove Theorem~\ref{9} making use of a classical submodule extension theorem of Siu and Trautmann~\ref{Siu Trautmann}.

In section ~\ref{comparison} we prove Theorem~\ref{8}. As we already pointed out,
we make crucial use of a Theorem~\ref{intro statement}  of Bungart~\cite[Theorem 8.1]{Bu} .

Finally, in section 9 we state some open problems.

The results in this  paper were  announced in \cite{KV}.

The second author wishes to thank Kari Astala, Bo Berndtsson, Laszlo Lempert, Eero Saksman, Bernard Shiffman, Andrei Suslin, and Hans-Olav Tylli for helpful conversations. The second author also thanks RIMS for hospitality and support.
Both authors thank the referee for constructive comments
which helped them to improve the exposition of this paper.

\section{Microdifferential operators}
\label{Microdifferential operators}

In this section we will recall the definition and basic properties of the sheaf of rings of microdifferential operators. We will also discuss the uniqueness part of the extension theorems.

Let $X$ be a complex manifold and let us write $\cO_X$ for its sheaf of holomorphic functions. We view it as a sheaf of topological rings in the following customary fashion. Let us fix an open subset $U\subset X$. For each compact subset $K\subset U$ we define a seminorm $\| \ \ \|_K$ on $\cO_X(U)$ as follows:
\begin{equation}
\|f\|_K \ = \ \sup_{x\in K} |f(x)| \qquad \text{for $f\in\cO_X(U)$.}
\label{eq:supnorm}
\end{equation}
Via these seminorms we equip $\cO_X(U)$  with a structure of a Fr\'echet space providing $\cO_X(U)$ with a structure of a topological ring. Throughout this paper we assume that $\cO_X$ has been given this structure. We also recall that $\cO_X(U)$ is a nuclear Fr\'echet space and hence $\cO_X$ is a sheaf of nuclear Fr\'echet rings. In the rest of this paper a ring on a topological space stands for a sheaf of rings.

We also write, as usual, $\cD_X$ for the sheaf of linear differential operators on $X$ with holomorphic coefficients. Let us now turn to  the rings of microdifferential operators, $\cE_X$ and $\hE_X$ (see \cite{SKK, Sch, K2}).
Here we will introduce them in terms of the symbol calculus in local coordinates.
For a coordinate free treatment see \cite{SKK}. We write $\ct$ for the cotangent bundle of $X$ and  $\pi_X\cl T^*X\to X$ for the projection.
The  $\bC^\times$-action
on $\ct$ gives rise to the Euler vector field $\chi$.
We say that a function $f(x,\xi)$ defined on an open subset of  $\ct$ is homogeneous of degree $j$
if $\chi f =j f$. Let us now consider a local symplectic
coordinate system $(x_1,\ldots,x_n;\xi_1,\ldots,\xi_n)$ of $T^*X$.
With this coordinate system, $\chi$ is written as $\sum_{i=1}^n\xi_i
\frac{\partial}{\partial\xi_i}$.
We define the sheaf $\hE_X(m)$ for $m\in\Z$ by setting, for an open subset $U$ of $T^*X$,
\begin{equation*}
\hE_X(m)(U) = \bigr\{\sum\limits_{j=-\infty}^m p_j(x,\xi)\mid
\text{$p_j(x,\xi)\in\cO_{T^*X}(U)$ is homogeneous of degree $j$}\bigr\}
\end{equation*}
and then set $\hE_X=\bigcup_{m\in\Z}\hE_X(m)$.
The expression $\sum_{j=-\infty}^m p_j(x,\xi)$
is to be viewed as a formal symbol.
The formal expressions are multiplied using the Leibniz rule:
\begin{equation}
\begin{gathered}
\text{For $p=\ssum_{i}p_i(x,\xi)$ and $q=\ssum_{i}q_i(x,\xi)$ we set $ pq=r=\ssum_{i}r_i(x,\xi)$}
\\
\text{where} \ \ \ \ r_k=\ssum_{k=i+j-|\alpha|} \frac{1}{\alpha!}
(\partial_\xi^\alpha p_i)(\partial_x^\alpha q_j)\,;
\end{gathered}
\end{equation}
here $\alpha=(\al_1,\ldots,\al_n)$ ranges over $\Z_{\ge0}^n$,
and $|\al|=\al_1+\cdots+\al_n$, $\al!=\al_1!\cdots\al_n!$,
$\partial_\xi^\alpha=(\frac{\partial}{\partial \xi_1})^{\al_1}\cdots
(\frac{\partial}{\partial \xi_n})^{\al_n}$.
 In this manner,
$\hE_X$ becomes a  ring on $\ct$.

We define $\cE_X$ to be the subsheaf of $\hE_X$
consisting of symbols $\sum_{j=-\infty}^m p_j(x,\xi)$
which satisfy the following growth condition:
\begin{equation}
\begin{gathered}
\text{for every compact $K\subset U$ there exists a $C>0$ such that}
\\
\sum_{j=-\infty}^0 \|p_j(x,\xi)\|_K \frac {C^{-j}}{(-j)!} <\infty\,;
\end{gathered}\label{definition of E}
\end{equation}
here  $\|p_j(x,\xi)\|_K$ stands for the sup norm on $K$  as in
\eqref{eq:supnorm}.
Standard estimates can be used to show that $\cE_X$ is
indeed closed under multiplication
and hence constitutes a subring of $\hE_X$. Often in this paper we call the case of $\cE_X$-modules the ``convergent'' case and
the case of $\hE_X$-modules the formal case.
The word ``convergence'' refers to the growth condition \eqref{definition of E} and not to actual convergence of  $\sum_{j=-\infty}^m p_j(x,\xi)$.

The sheaves $\cE_X$ and $\hE_X$
are coherent and Noetherian rings on $T^*X$, see \cite{SKK}. Furthermore, coherent modules over $\cE_X$ and $\hE_X$
are supported on  analytic subvarieties of $\ct$. By a fundamental result
\begin{equation}
\text{the support of coherent $\cE_X$- and $\hE_X$-modules is involutive}\,.
\end{equation}
For a proof see \cite{SKK,K2}. Recall that coherent  $\cE_X$- and $\hE_X$-modules whose support is Lagrangian are called {\em holonomic}.
In the study of both $\cE$-modules and $\hE$-modules,
we can make use of quantized contact transformations. A contact transformation between two open sets $U\subset \ct$ and $V\subset T^*Y$ is a biholomorphic map $\phi\cl U \to V$ such that $\phi^*\alpha_Y = \phi^* \alpha _X$, where $\alpha_X$ and $\alpha_Y$ are the canonical 1-forms of $\ct$ and $T^*Y$, respectively.
A contact transformation can be quantized, at least locally.
In other words, any point in $U$ has a neighborhood $W$
such that there exists an isomorphism of $\C$-algebras
between $\cE_X|_W$ and $(\phi^{-1}\cE_Y)|_{W}$. Thus, for local questions concerning $\cE_X$-modules, we can make use of contact  transformations and put the characteristic variety in a convenient position.
Recall that we say that  a conic Lagrangian variety $\La$
is {\em in generic position} at a point $p\in\La$ if the the fibers of the projection $\La \to X$ are at most one dimensional in the neighborhood of $p$.

This condition can be spelled out concretely in local coordinates
in the following manner. Let us write  $(x_1,\ldots,x_n;\xi_1,\ldots,\xi_n)$
for local coordinates where the $x_i$ are the coordinates on the base $X$
and the $\xi_i$ are the corresponding fiber coordinates.
 We assume that $p$ is the point $(0,\ldots,0\,;0,\ldots,0,1)$,
i.e., that $p$ is $dx_n$ at the origin.
If $\La$ is in generic position at $p$ then
\begin{equation}
\La \ = \  {T^*_SX} \quad \text{where} \ \ S= \{f=0\} \ \ \ f = x_n^k + \text{h.o.t.\,,}
\end{equation}
where h.o.t. stands for a holomorphic function in
the ideal $(x_1,\ldots,x_n)^{k+1}$.
Note that ${T^*_SX}$ stands for  the closure of
 the conormal bundle
$T^*_{S\setminus S_\sing}(X\setminus S_\sing)$ where $S_\sing$ is the singular locus of $S$.

It is not very difficult to see that a conic Lagrangian variety can always locally be put in a generic position via a contact transformation, see, for example, \cite[Corollary 1.6.4]{KK}. Hence, for local questions about holonomic modules we can always assume that the characteristic variety is in general position.
\begin{rmk}
When working with $\cE$-modules we always assume that we work outside of the zero section. This is not a serious restriction as we can always add a
{``dummy variable''} to $X$.
\end{rmk}

One way to justify the convergence condition for symbols in $\cE_X$ is the following basic fact:
\begin{thm}
Let us assume that  the support of a holonomic $\cE_X$-module $\cM$
is in generic position at a point $p\in\ct$.
Then the local $\cE_{X,p}$-module $\cM_p$,
the stalk of\/ $\cM$ at the point $p$,
is a holonomic $\cD_{X,\pi_X(p)}$-module,
and the canonical morphism $\cE_{X,p}\otimes_{\cD_{X,\pi_X(p)}}\cM_p\to \cM_p$
is an isomorphism.
\end{thm}

In particular, local questions about $\cE_X$-modules can be reduced to questions about $\cD_X$-modules.
A proof of this result can be found in \cite[Theorem 8.6.3]{Bj}, for example.
The proof uses the same reduction technique which we utilize in this paper combined with Fredholm theory. The reduction technique is explained in section \ref{The reduction} of this paper.
The estimates in definition \eqref{definition of E}  are precisely the ones for the theorem above to hold.

Let us recall the notion of regular singularities.
For a coherent $\cE_X$-module $\cM$,
a coherent $\cE_X(0)$-submodule $\cN$ is called an $\cE_X(0)$-{\em lattice}
if $\cE_X\otimes_{\cE_X(0)}\cN\to\cM$ is an isomorphism.
A holonomic $\cE$-module with support $\La$
is said to have {\em regular singularities} or be {\em regular\/}
if locally near any point on the support of $\cM$
the module $\cM$ has an $\cE(0)$-lattice $\cN$
which is invariant under $\cE_\La(1)$,
the subsheaf of order 1 operators whose principal symbol vanishes on $\La$.
Kashiwara and Kawai show, using their notion of order:

\begin{thm}
A regular holonomic  $\cE$-module possesses
a globally defined $\cE(0)$-lattice invariant under  $\cE_\La(1)$.
The analogous result holds for $\hE$-modules.
\end{thm}

For a proof see \cite[Theorem 5.1.6]{KK}.
In the rest of the paper we make use of the (global)
existence of an $\cE(0)$-lattice.
Its invariance under $\cE_\La(1)$ will play no role.
For the rest of this paper we assume that
our holonomic $\cE_X$-modules possess an $\cE_X(0)$-lattice.
 We do  not know if this is true for all holonomic modules.

Let us now consider the question of uniqueness in the convergent and formal versions of the codimension three extension Theorems~\ref{1} and~\ref{3} and of the  submodule extension Theorems~\ref{2} and ~\ref{4}. Recall that we are considering an open subset $U$ of $T^*X$, $\Lambda$ a closed Lagrangian
analytic subset of $U$, and $Y$ a closed analytic subset of $\Lambda$
of codimension {\em at least two}.
We will also write $j\cl U\setminus Y\imbed U$ for  the open inclusion. We let
$\cM$ (respectively $\hM$) be a holonomic $\bl\cE_X\vert_{U\setminus Y}\br$-module (respectively $\bl\hE_X\vert_{U\setminus Y}\br$-module)
on the open subset $U$ of $T^*X$,
whose support is contained in $\Lambda\setminus Z$.

We argue first that if an extension of  $\cM$ (respectively $\hM$) to $U$ with support in $\La$ exists then they are unique and they coincide with the sheaf extension $j_*\cM$ ($j_*\hM$, respectively). As the arguments are the same in the convergent and the formal cases we will just work with the convergent case. We first recall that holonomic modules are Cohen-Macaulay,
i.e., a module $\cM'$ is holonomic if and only if
\begin{equation}
\cext^k_{\cE_X}(\cM',\cE_X) = 0 \qquad \text{unless} \ \ k=\dim X\,.
\end{equation}
and of course we have the duality statement
\begin{equation}
\cext^n_{\cE_X}(\cext^n_{\cE_X}(\cM',\cE_X),\cE_X) \simeq\cM'\,.
\end{equation}
This implies:
\eq
\label{CM condition}
&&\parbox{70ex}{If $\cM'$ be a holonomic $(\cE_X\vert_U)$-module supported
in a Lagrangian variety $\Lambda\subset U$, then we have
$\cH^k_Z(\cM')=0$ for $k<\codim_\Lambda Z$
for any closed analytic subset $Z$ of $\Lambda$.}
\eneq
Let us now assume that $\cM'$ is an extension of $\cM$ to $U$ with support $\La$. Let us write $i\cl Y\to U$ for  the closed inclusion. Then we have the following exact triangle:
\begin{equation*}
\to Ri_*i^! \cM' \to \cM' \to Rj_*\cM \to
\end{equation*}
{}From \eqref{CM condition} we conclude that $\cM'\cong j_*\cM$, i.e., that the extension is unique, as long as $\codim_\Lambda Y\geq 2$. Thus, we are reduced to proving that  $j_*\cM$ and $j_*\hM$ are holonomic. Note that the holonomicity of  $j_*\cM$ and $j_*\hM$ amounts to them being coherent.  In this situation the sheaves  $j_*\cM$ and $j_*\hM$ would fail to be coherent if they do not have sufficiently many sections on $Y$, for example, if the restrictions of $j_*\cM$ and $j_*\hM$ to $Y$ were to be zero.

\section{Construction of commutative rings}

In this section we introduce  commutative rings on complex manifolds
that will be important for us. These rings are simpler versions of the rings $\cE_X(0)$ and $\hE_X(0)$.  In the next section explain the relationship between the extension theorems for microdifferential operators and the extensions theorems for our commutative rings which were stated in the introduction.

Consider the formal power series ring $\hA = \bC[[t]]$. It is a discrete valuation ring. We define a subring $A$ of $\hA$ in the following manner.
For any $C>0$ we define a norm $\|\ \|_C$ on $\hA$ by the formula
\begin{equation}
\label{C norm}
\|\sum_{j=0}^\infty a_jt^j\|_{_C}\ = \ \sum_{j=0}^\infty |a_j|\frac{C^j}{j!}
\,.
\end{equation}
We write $A_C$ for the subring consisting of elements $a$ of $\hA$ with
$\| a\| _C<\infty$. The ring $A_C$ is a Banach local ring as can be concluded from the following lemma.
\Lemma
\label{lem:Banachring}
\bnum
\item
For any $a$, $b\in A_C$, we have
$$\|ab\|_C\le \|a\|_C\;\|b\|_C.$$
More generally if $a\in A_C\cap t^pA$
$b\in A_C\cap t^qA$ for $p,q\in\Z_{\ge0}$, then
$$\|ab\|_C\le \dfrac{p!q!}{(p+q)!}\|a\|_C\;\|b\|_C.$$
\item If $a\in A_C\cap tA$, then $1-a$ is an invertible element in $A_C$.
\enum
\enlemma
\Proof
Set $a=\sum_{j\ge p}a_jt^j$,
$b=\sum_{j\ge q}b_jt^j$ and
$c\seteq ab=\sum_{j\ge p+q}c_jt^j$.
Then we have
$c_k=\sum_{k=i+j}a_ib_j$, and
\eqn
\|c\|_C&\le&\sum_{k=i+j,\, i\ge p,\,j\ge q}\dfrac{C^k}{k!}|a_i|\;|b_j|
=\sum_{i\ge p,\,j\ge q}C^{i+j}\dfrac{i!j!}{(i+j)!!}
\dfrac{|a_i|}{i!}\dfrac{|b_j|}{j!}\\
&\le&\dfrac{p!q!}{(p+q)!}
\sum_{i,j}C^{i+j}\dfrac{|a_i|}{i!}\dfrac{|b_j|}{j!}
=\dfrac{p!q!}{(p+q)!}\|a\|_C\|b\|_C.
\eneqn
(ii) For $a\in A_C\cap tA$, we have
$\|a^n\|_C\le \|a\|_C^n/n!$ and hence $1-a$ is invertible in $A_C$.
\QED

Finally, we set
\begin{equation}
A = \varinjlim _{C \to 0}A_C\,.
\end{equation}

\begin{prop}
\label{structure of A}
The topological ring $A$ is a dual nuclear Fr\'echet discrete valuation ring.
\end{prop}

\begin{proof}
Clearly, the maps $A_C \to A_D$, for $D<C$, are nuclear. Thus, $A$ is a DNF algebra. It remains to show that $A$ is a discrete valuation ring. This follows from the statement:
\begin{equation*}
\parbox{60ex}
{Any non-zero element $a\in A$ can be written as $a=ut^\ell$
with an invertible element $u$ of in $A$.}
\end{equation*}
To prove this, let $a=\sum_{k=\ell}^\infty a_k t^k\in A$ with $a_\ell\neq 0$. Then there exists a $C>0$ such that $\sum_{k=\ell}^\infty |a_k| {C^k}/{k!}<\infty$.
We now write
\begin{equation*}
a \ = \ ut^\ell \quad \text{with $u= \sum_{k=0}^\infty a_{k+\ell} t^k$.}
\end{equation*}
For $u$ to lie in $A$ there has to exist a $D$ such that
\begin{equation*}
 \sum_{k=0}^\infty |a_{k+\ell}| \frac {D^k}{k!}<\infty.
\end{equation*}
But now
\begin{equation*}
\ba{l}
 \sum_{k=0}^\infty |a_{k+\ell}| \dfrac {D^k}{k!}
=C^{-\ell}\sum_{k=0}^\infty |a_{k+\ell}| \dfrac {C^{k+\ell}}{(k+\ell)!}\dfrac{(k+\ell)!}{k!}
\Bigl(\dfrac{D}{C}\Bigr)^k\\[1ex]
\hs{10ex}\leq C^{-\ell}\left(\sum_{k=0}^\infty |a_{k}| \dfrac {C^{k}}{k!} \right)^{1/2}
\left(\sum_{k=0}^\infty \dfrac{(k+\ell)!}{k!}\Bigl(\dfrac{D}{C}\Bigr)^k\right)^{1/2}.
\ea
\end{equation*}
The last series converges as long as $D< C$.
\end{proof}

We write $K$ for the fraction field of $A$
and $\hK$ for the fraction field of $\hA$.
Then $K= A[t^{-1}]$ and $\hK= \hA[t^{-1}]$;
the field $\hK$ is thus the field of formal Laurent series.
Note that we can identify $K$ with the  subring of
constant coefficient operators in
$\cE_\bC$ and $\hK$ with  the subring of
constant coefficient operators in $\hE_\bC$
by identifying $t$ with $\bl\frac d {dx}\br^{-1}$.
Under this identification $A$ corresponds to
 the subring of  constant coefficient operators in $\cE_\bC(0)$
and $\hA$ to the subring of  constant coefficient operators
in $\hE_\bC(0)$.

Let $X$ be a complex manifold.
We write $\cA_X$ for the sheaf of holomorphic functions on $X$
with values in $A$ and similarly for $\hcA_X$, $\cK_X$, and $\hcK_X$.
We can also view $\cA_X$ as a (projective) topological tensor product
$\cA_X = A \hat\otimes_\bC \cO_X$,
and similarly for  $\hcA_X$, $\cK_X$, and $\hcK_X$.
If we denote by $\cA_X^C$ the sheaf of holomorphic functions on $X$
with values in $A_C$, then
we have
$$\cA_X \simeq\indlim[C]\cA_X^C.$$
We have also isomorphisms
$\cK_X\simeq\cA_X[t^{-1}]$ and $\hcK_X\simeq\hcA_X[t^{-1}]$ and  of course $\hcA_X\cong \cO_X[[t]]$.

Note that $\cA_X$, $\hcA_X$, $\cK_X$ and $\hcK_X$ are Noetherian rings on $X$.

\medskip
Let $\cM$ be a coherent $\cK_X$-module. We say that a coherent $\cA_X$-submodule $\cN$ of $\cM$ is an $\cA_X$-{\em lattice} if we have an isomorphism
\begin{equation}
\cK_X \otimes_{\cA_X} \cN\isoto\cM\,.
\end{equation}

Note that lattices are always $t$-torsion free as $\cN$ is a submodule of $\cM$. Similarly, we define the notion of an $\hcA_X$-lattice in a coherent $\hcK_X$-module.

 \section{Reduction via finite morphisms}
  \label{The reduction}

In this section we discuss the relationship between the microdifferential  and commutative versions of our extension theorems.
We do so by a standard technique in several complex variables, due in this context to \cite{SKK}. Via this technique we are able to ``eliminate'' the extraneous variables and reduce the microlocal extension problems to the commutative versions.

In this section it is more convenient to work in the projectivized setting. We consider the projective cotangent bundle
$\bP^*X\seteq\dT X/\C^\times$  where $\dT X\seteq T^*X\setminus X$.
Since $\cE_X$ and $\hE_X$ are constant along the fibers of
$\dT X\to \bP^*X$, we regard $\cE_X$ and $\hE_X$
as rings on $\bP^*X$.

Let $\Omega$ be the open subset
$\set{(x;\xi)\in\bP^*\C^n}{\xi_n\not=0}$ of $\bP^*\C^n$,
and let $\rho\cl\Omega\to\C^{n-1}$ be the map defined by
$\rho(x;\xi)=(x_1,\ldots,x_{n-1})$.
Let $X'$ be an open subset of $\C^{n-1}$
and $\Omega'=\rho^{-1}X'$.
Let $\rho'\cl \Omega'\to X'$ denote the restriction of $\rho$.

Then $\rho'_*(\cE_{\C^n}\vert_{\Omega'})$ contains $\cK_{X'}$ by sending
$t$ to $\partial_n^{-1}$.
Similarly $\rho'_*(\cE_{\C^n}(0)\vert_{\Omega'})$ contains $\cA_{X'}$,
and similarly for the formal case.
Let us denote by $\DA_{X'}$ the subring of
$\rho'_*(\cE_{\C^n}(0)\vert_{\Omega'})$ generated by $\A_{X'}$,
the $t\partial_k$  for $k=1,\ldots,n-1$, and $x_n$.
Similarly we denote by $\DK_{X'}$ the subring of
$\rho'_*(\cE_{\C^n}\vert_{\Omega'})$ generated by $\cK_{X'}$,
the $\partial_k$ for $k=1,\ldots,n-1$, and $x_n$.
Then we have
\begin{equation}
\label{poly}
\begin{gathered}
\DA_{X'}\simeq\cA_{X'}\otimes_\C\C[t\partial_1,\ldots,t\partial_{n-1},x_n],\\
\DK_{X'}\simeq\cK_{X'}\otimes_{\cA_{X'}}\DA_{X'}
\simeq\cK_{X'}\otimes_\C\C[\partial_1,\ldots,\partial_{n-1},x_n].
\end{gathered}
\end{equation}
We define their formal analogues
$\hDA_{X'}$ and $\hDK_{X'}$ in the same fashion.

Let $\Mod_\gd(\DK_{X'})$ denote the abelian category of
coherent $\DK_{X'}$-modules $\cN$
such that there exists locally
a coherent $\DA_{X'}$-submodule $\cL$ of $\cN$
satisfying the two conditions:
\eqn
&&\cN\simeq(\DK_{X'})\otimes_{(\DA_{X'})}\cL,
\\
&&\text{$\cL$ is a coherent $\cA_{X'}$-module.}
\eneqn
Let us denote by $\Mod_{\rho-\gd}(\cE_{\C^n}\vert_{\Omega'})$
the category of coherent
$\cE_{\C^n}\vert_{\Omega'}$-modules $\cM$ such that
the support of $\cM$ is finite over $X'$.
We define their formal analogues
$\Mod_\gd(\hDK_{X'})$ and $\Mod_{\rho-\gd}(\hE_{\C^n}\vert_{\Omega'})$ similarly.
Note that, for dimension reasons, the modules $\cM$ are holonomic.

Below we state two propositions which are analogues of  classical theorems on finite morphisms in several complex variables. The first one concerns behavior of coherence under finite morphisms and the  second one
is an analogue of the duality theorem for finite morphisms for analytic coherent sheaves.
\Prop\label{prop:equivK}
The abelian categories $\Mod_{\rho-\gd}(\cE_{\C^n}\vert_{\Omega'})$ and $\Mod_{\rho-\gd}(\hE_{\C^n}\vert_{\Omega'})$ are
equivalent to  the abelian categories $\Mod_\gd(\DK_{X'})$ and $\Mod_\gd(\hDK_{X'})$, respectively, via
the functors
$\cM\longmapsto \rho'_*(\M)$ and $\hM\longmapsto \rho'_*(\hM)$, respectively. Their quasi-inverses are given by
$\cN\longmapsto (\cE_{\C^n}\vert_{\Omega'})\otimes_{\DK_{X'}}\rho'^{-1}(\cN)$ and $\hN\longmapsto (\hE_{\C^n}\vert_{\Omega'})\otimes_{\hDK_{X'}}\rho'^{-1}(\hN)$, respectively.
\enprop
and
\Prop\label{Serre duality variant}
For $\cM\in\Mod_{\rho-\gd}(\cE_{\C^n}\vert_{\Omega'})$ and for $\hM\in\Mod_{\rho-\gd}(\hE_{\C^n}\vert_{\Omega'})$ we have the duality isomorphisms
$${\rho'}_*\bl\cext^k_{\cE_{\C^n}}(\cM,\cE_{\C^n}\vert_{\Omega'})\br\simeq
\cext^{k-n}_{\cK_{X'}}({\rho'}_*(\cM),\cK_{X'})\,.
$$
$${\rho'}_*\bl\cext^k_{\hE_{\C^n}}(\hM,\hE_{\C^n}\vert_{\Omega'})\br\simeq
\cext^{k-n}_{\hcK_{X'}}({\rho'}_*(\hM),\hcK_{X'})\,.
$$
\enprop

As similar  statements are proved in  \cite[Chapter 3]{SKK} and the proofs proceed exactly in the same manner as in the classical case we just briefly indicate the outline of the arguments. We discuss only the microdifferential case as the argument in the formal case is the same.
The idea  is
to compose the projection $\rho$ into projections where the fiber is one dimensional and then proceed step by step.  Thus, we consider a projection\footnote{In an analogous manner we can consider projections of the type $\tau\cl \bC^{k} \times (\bP^{\ell+1}-\{\infty\}) \to  \bC^{k} \times \bP^{\ell}$.}
\begin{subequations}
\begin{equation}
\tau\cl \bC^{k+1} \times \bP^{\ell} \to  \bC^{k} \times \bP^{\ell}
\end{equation}
given by
\begin{equation}
\tau(x_1,\dots,x_{k+1};\xi_{n-\ell},\dots,\xi_n)\
 = \ (x_1,\dots,x_{k};\xi_{n-\ell},\dots,\xi_n)\,.
\end{equation}
\end{subequations}
We write $\cE_{ \bC^{k} \times \bP^{\ell}}$ for the sheaf of microdifferential operators on $\bC^{k} \times \bP^{\ell}$.
It is naturally a subsheaf of $\cE_{\bC^n}$
where the symbols just depend on the variables
$(x_1,\dots,x_{k};\xi_{n-\ell},\dots,\xi_n)$.
In this language $\cK_{\bC^n} = \cE_{ \bC^{n} \times \bP^{0}}$.
Then $\cE_{ \bC^{k} \times \bP^{\ell}}$ can be naturally identified with a subsheaf of
$\tau_*\cE_{ \bC^{k+1} \times \bP^{\ell}}$.
We also write $\cE_{\bC^{k} \times \bP^{\ell}}(0)[x_{k+1}]$
for the subsheaf of $\tau_*\cE_{ \bC^{k+1} \times \bP^{\ell}}$
generated by $\cE_{ \bC^{k} \times \bP^{\ell}}(0)$ and $x_{k+1}$.
Then Propositions~\ref{prop:equivK}
follow from:
\Lemma
Let\/ $V$ be an open subset of $\bC^{k}\times \bP^{\ell}$.
Write $\Mod_{\tau\text{-fin}}(\cE_{\bC^{k+1} \times \bP^{\ell}}(0)
\vert_{\tau^{-1}V})$
for the category of coherent
$\cE_{\bC^{k+1} \times \bP^{\ell}}(0)\vert_{\tau^{-1}V}$-modules
$\cN$ such that
$\Supp(\cN)\to V$ is a finite morphism,
and
write $\Mod_\coh(\cE_{\bC^{k} \times \bP^{\ell}}(0)[x_{k+1}]\vert_V)$
for the category of
$\cE_{\bC^{k} \times \bP^{\ell}}(0)[x_{k+1}]\vert_V$-modules $\cL$
that are coherent over $\cE_{\bC^{k} \times \bP^{\ell}}(0)\vert_V$.
Then the functor $\cN\mapsto \tau_*\cN$ gives an equivalence of
categories between
$\Mod_{\tau_\text{fin}}(\cE_{\bC^{k+1} \times \bP^{\ell}}(0)
\vert_{\tau^{-1}V})$ and
$\Mod_\coh(\cE_{\bC^{k} \times \bP^{\ell}}(0)[x_{k+1}]\vert_V)$.
\enlemma

As we stated before, this lemma is proved in the same manner as the statements for coherent analytic sheaves
making use of the Weierstrass preparation theorem and division theorems.
The Weierstrass preparation theorem and the division theorems are proved
in \cite[Chapter 2]{SKK}.  Proposition~\ref{Serre duality variant} is proved in the similar manner.

Furthermore, arguing as above, we have

\Prop
\label{coherence from finite morphisms}
\bnum
\item
Let $\cN$ be an $\cE_{\C^n}(0)\vert_{\Omega'}$-module
and assume that\linebreak $\rho'\vert_{\Supp(\cN)}\cl\Supp(\cN) \to X'$
is finite. Then $\cN$ is a coherent $\cE_{\C^n}(0)\vert_{\Omega'}$-module
if and only if $\rho'_*\cN$ is a coherent $\cA_{X'}$-module.
\item
Similarly let $\cM$ be an $\cE_{\C^n}\vert_{\Omega'}$-module
such that $\rho'\vert_{\Supp(\cM)}\cl\Supp(\cM)\to X'$
is finite.
\bna
\item
If $\cM$ is a coherent $\cE_{\C^n}\vert_{\Omega'}$-module
then $\rho'_*\cM$ is a coherent $\cK_{X'}$-module,
\item If $\rho'_*\cM$ is a coherent $\cK_{X'}$-module
and if $\rho'_*\cM$ has a coherent $\cA_{X'}$-lattice $\cL$
such that $\cL$ is an $\DA_{X'}$-submodule,
then $\cM$ is a coherent $\cE_{\C^n}\vert_{\Omega'}$-module.
\ee
\ee
\enprop
Of course we have a similar statement in the formal case.
 \begin{rmk}
We do not know if (ii) (b) holds without assuming the existence of a
lattice $\cL$. Having the stronger statement available would slightly simplify some of our arguments.
 \end{rmk}

We now analyze the coherent $\cK_{X'}$-module ${\rho'}_*(\cM)$.
We claim:
\begin{prop}
\label{vector bundle}
The sheaf ${\rho'}_*(\cM)$ is locally free over $\cK_{X'}$ and the sheaf  ${\rho'}_*(\hM)$ is locally free over $\hcK_{X'}$.
\end{prop}

In other words ${\rho'}_*(\cM)$ and  ${\rho'}_*(\hM)$ are holomorphic vector bundles of finite rank over
the fields $K$ and $\hK$, respectively.

\begin{proof}
We give the proof in the convergent case, in the formal case the proof is exactly the same.
To argue this, we first recall that holonomic modules are Cohen-Macaulay,
i.e., we have
\begin{equation}
\cext^k_{\cE_X}(\cM,\cE_X) = 0 \qquad \text{unless} \ \ k=\dim X\,.
\end{equation}

{}From Proposition~\ref{Serre duality variant}
 we conclude
\begin{equation}\label{eq:van}
\cext^k_{\cK_{X'}}({\rho'}_*(\cM),\cK_{X}) = 0
\quad\text{unless $k=0$.}
\end{equation}
Taking the germs at $x'\in X'$, we obtain
\begin{equation}\label{eq:vangerm}
\Ext^k_{\cK_{X',x'}}\bl({\rho'}_*(\cM))_{x'},\cK_{X',x'}\br = 0
\quad\text{unless $k=0$.}
\end{equation}
Let us now consider the ring $\cK_{X',x'}$. Let us first note that
\begin{subequations}
\begin{equation}
\text{the ring $\cA_{X',x'}=  A \hat\otimes_\bC \cO_{X',x'}$
is a commutative regular local ring}
\end{equation}
and that
\begin{equation}
\cK_{X',x'}= \cA_{X',x'}[t^{-1}].
\end{equation}
\end{subequations}
In particular $\cK_{X',x'}$ is a Noetherian ring with
 finite global dimension.
This along with \eqref{eq:vangerm} implies that
\begin{equation}
\text{$({\rho'}_*(\cM))_{x'}$ is a finitely generated projective
$\cK_{X',x'}$-module.}
\end{equation}

We now make use of the following theorem
of Popescu, Bhatwadekar, and Rao \cite{P,BR}; for a nice discussion, see also \cite{S}. They show:
\begin{equation}
\parbox{70ex}{Let $R$ be a regular local ring containing a field
with maximal ideal $\fm$ and $t\in\fm\setminus\fm^2$.
Then every finitely generated projective module
over the localized ring $R_t=R[t^{-1}]$ is free.}
\end{equation}
This result is related to Serre's conjecture and was conjectured by Quillen in \cite{Q}.

Hence $({\rho'}_*(\cM))_{x'}$ is a free $\cK_{X',x'}$-module
for any $x'\in X$, and we thus finally conclude
\begin{equation}\label{eq:free}
\text{${\rho'}_*(\cM)$ is a locally free $\cK_{X'}$-module
of finite rank.}
\end{equation}

\end{proof}
\bigskip

 \section{Proof of the main theorems}
  \label{The structure of the proof}

In this section we prove our main extension theorems. We first prove the formal versions Theorems~\ref{3} and
~\ref{4} by reducing them to their formal commutative analogues Theorems~\ref{5} and ~\ref{9} by methods of section~\ref{The reduction}.
We then prove Theorems~\ref{1} and ~\ref{2}
making use of the formal Theorems ~\ref{3} and ~\ref{4}
we just proved and the comparison Theorem~\ref{8}. The proofs of Theorems~\ref{5}, \ref{9}, and~\ref{8} are postponed and their proofs are given in their own sections~\ref{The formal case}, \ref{The formal submodule theorem}, and \ref{comparison}, respectively.

\subsection{General preliminaries} \label{General}In this subsection we make some preliminary constructions which will be used in all of the arguments.

Let us recall the setup common to all of the microlocal extension theorems. We consider an open subset $U$ of $T^*X$, $\Lambda$ a closed Lagrangian
analytic subset of $U$, and $Y$ a closed analytic subset of $\Lambda$.  As we remarked earlier, all the extension theorems are local in nature. Thus, as far as the microlocal extension theorems are concerned, we can assume that we work in the vicinity of a point $p\in Y$. Furthermore, working inductively, we can assume that the point $p$ is a smooth point in $Y$. In addition making use of a quantized contact transformation and the generic position lemma,  \cite[Corollary 1.6.4]{KK}, we can assume that the characteristic variety $\La$ is in generic position at $p$.

We will make use of the results of section~\ref{The reduction}.
In that section we worked in the projectivized setting. We consider the projective cotangent bundle
$\bP^*X\seteq\dT X/\C^\times$ where $\dT X\seteq T^*X\setminus X$.
Since $\cE_X$ and $\hE_X$ are constant along the fibers of
$\dT X\to \bP^*X$, we regard $\cE_X$ and $\hE_X$
as rings on $\bP^*X$.
 We will now also regard the Lagrangian $\La$ as a
locally closed subvariety of $\bP^*X$. We also continue to denote by $Y$ the projectivization of the original $Y$ in $\ct$ and similarly for the open set $U$.

 As $\La$
in generic position at $p$ we can make the following choice of local coordinates.
In the neighborhood of $\pi_X(p)$ we choose
a local coordinate system $(x_1,\ldots, x_n)$
such that $\pi_X(p)$ corresponds to the origin
and the point $p$ corresponds to $dx_n$ at the origin.
Thus we may assume that $X$ is an open subset of $\C^n$.
By shrinking $U$ if necessary,
we can assume that $U$ is contained in
$\Omega\seteq\set{(x;\xi)\in \bP^*\bC^n}{\xi_n\not=0}$
and the restriction
$\rho\vert_U\cl U \to \bC^{n-1}$ of $\rho\cl\Omega\to\C^{n-1}$
has the following properties:
\begin{subequations}
\begin{equation}
\text{$X'\seteq\rho(U)\subset\bC^{n-1}$
is an open neighborhood of $\rho(p)\in\bC^{n-1}$,}
\end{equation}
\begin{equation}
\text{$\rho\vert_{\La}\cl \La \to X'\subset\bC^{n-1}$ is finite,}
\end{equation}
\begin{equation}
\text{$\rho|_{Y} \cl Y \to Y'\seteq\rho(Y)$
is an isomorphism.}
\end{equation}
\end{subequations}
In particular, $Y'$ is then a smooth submanifold of $X'$.
of the same codimension as $Y$ in $\Lambda$.
By abuse of notation we will now simply write $\rho$ for $\rho|_U$.
By replacing $U$ with $\rho^{-1}(X')$
we may assume that
$U=\rho^{-1}(X')$.
We also write $\dU=U\setminus  \rho^{-1}Y'$ and
$$\drho\cl \dU\to X'\setminus Y'$$
for the restriction of $\rho$ to $\dU$.

We recall that we have written $j\cl U\setminus Y \hookrightarrow U$ for the inclusion and
we will write $j'\cl X'\setminus Y'\hookrightarrow X'$ for the other inclusion. We summarize the situation in the following commutative Cartesian diagram
\begin{equation}
\begin{gathered}
\xymatrix{
 {\dU}\ar[d]_\drho \ar@{^{(}->}[r]^j \ar@{}[dr]|{\square}&
{U}
\ar[d]^\rho\\
 {X'\setminus Y'} \ar@{^{(}->}[r]^{\ \ \ j'}
& {X'}\,.
}
\end{gathered}
\end{equation}

\subsection{The formal codimension three extension theorem}\label{Formal}

In this subsection we prove Theorem~\ref{3}
 by reducing it to its commutative version Theorem~\ref{5}
whose proof is given later in section \ref{The formal case}.

We work in the geometric setting of subsection~\ref{General}
with $U$ an open subset of $T^*X$, $\Lambda$ a closed Lagrangian
analytic subset of $U$, and $Y$ a closed analytic subset of $\Lambda$
of codimension  at least three. We are given a holonomic  $\bl\hE_X\vert_{U\setminus Y}\br$-module $\hM$
whose support is contained in $\Lambda\setminus Y$ and an
 $\bl\hE_X(0)\vert_{U\setminus Y}\br$-lattice $\hN$ of $\hM$. We write $j\cl U\setminus Y\imbed U$ for the open inclusion and
we are to show that $j_*\hM$ is a coherent $(\hE_X|_U)$-module.

By replacing the lattice $ \hN$ with the lattice
$$\cext_{\hE_X(0)\vert_{U\setminus Y}}^n
\bl\cext_{\hE_X(0)\vert_{U\setminus Y}}^n(\hN,\hE_X(0)\vert_{U\setminus Y}),
\hE_X(0)\vert_{U\setminus Y}\br\,,
$$
we may assume from the beginning that
\eq
&&\hN\simeq \cext_{\hE_X(0)\vert_{U\setminus Y}}^n
\bl\cext_{\hE_X(0)\vert_{U\setminus Y}}^n(\hN,\hE_X(0)\vert_{U\setminus Y}),
\hE_X(0)\vert_{U\setminus Y}\br.\label{eq:hNref}
\eneq
 As the question is local, we can proceed inductively along $Y$ and so we can assume that we work in a neighborhood of a smooth point $p\in Y$.
Furthermore, we shrink the open set $U$
as in the  subsection~\ref{General} above.
We now make use of  Propositions~\ref{coherence from finite morphisms}
and  ~\ref{vector bundle}
to conclude that $\drho_*(\hM\vert_{\dU})$ is a
locally free $\hcK_{X'\setminus Y'}$-module of finite rank,
and $\drho_*(\hN\vert_{\dU})$  is a
coherent $\hcA_{X'\setminus Y'}$-lattice of $\drho_*(\hM\vert_{\dU})$.

Then \eqref{eq:hNref} implies that $\drho_*(\hN\vert_{\dU})$ is
a reflexive coherent $\cA_{X'\setminus Y'}$-module
by Proposition \ref{Serre duality variant} and Proposition~\ref
{coherence from finite morphisms}.
Hence, we can apply Theorem~\ref{5} and conclude
that $j'_*\drho_*(\hN\vert_{\dU})$ is  a coherent $\hcA_{X'}$-module.
We now note that
\begin{equation}
\rho_* j_*\hM \ \cong j'_*\drho_*(\hM\vert_{\dU}) \qquad \text{and} \qquad \rho_*j_*\hN \ \cong j'_*\drho_*(\hN\vert_{\dU})\,.
\end{equation}
 Now we again apply Proposition~\ref{coherence from finite morphisms},
and as  $j'_*\drho_*(\hN\vert_{\dU})$ is a coherent $\hcA_{X'}$-module,
we conclude that $j_*\hN$ is a coherent $\hE_X(0)|_U$-module and $j_*\hM$ is a coherent $\hE_X|_U$-module.

\subsection{The formal codimension two submodule extension theorem}
\label{formal submodule theorem proof}

In this subsection we prove Theorem~\ref{4}
 by reducing it to its commutative version Theorem~\ref{9}
whose proof is given later in section
\ref{The formal submodule theorem}.

It is more convenient for us to work with the equivalent quotient version of the theorem:
 \begin{thm}
\label{formal quotient module theorem}
Let $U$ be an open subset of $T^*X$,
$Y$ an analytic subset of $U$ of codimension two or more,
and $j\cl U\setminus Y \to Y$ the inclusion.
Let $\hM$
be a holonomic  $\hE_X\vert_U$-module.
Let us assume that we are given a holonomic $\hE_X\vert_{U\setminus Y}$-module
$\hM_2$ which is a quotient of  $j^{-1}\hM$.
Then $\on{Im}(\hM\to j_*\hM_2)$
is a coherent  $(\hE_X\vert_U)$-module on $U$.
 \end{thm}

We work in the geometric setting of subsection~\ref{General}. Thus, we are working in a neighborhood of smooth point $p\in Y$ where $Y$ is of codimension at least two. Furthermore, we shrink the open set $U$ as in the subsection~\ref{General}. We choose an $\hE_X(0)\vert_U$-lattice $\hN$ of $\hM$. To be able to make this choice we might have to shrink $U$ further.

 We now make use of Propositions~\ref{vector bundle} to conclude that
$\rho_*(\hM)$ is a locally free $\hcK_{X'}$-module of finite rank.
Furthermore, $\rho_*(\hN)$ is  a coherent $\hcA_{X'}$-lattice of
$\rho_*(\hM)$. We also conclude that
$\drho_*(\hM_2\vert_{\dU})$   is a locally free
$\hcK_{X'\setminus Y'}$-module of finite rank.

Because $\hM_2$  is a quotient of  $j^{-1}\hM$,
we have a morphism of $\rho_*(\hM)\vert_{X'\setminus Y'} \to \drho_*(\hM_2\vert_{\dU})$ and we let $\hL$ be the image of $\drho_*(\hN)\vert_{\dU}$ under this morphism.
As $\hL$ lies in the locally free $\hcK_{X'\setminus Y'}$- module  $\drho_*(\hM_2\vert_{\dU})$   it is torsion free. Now can apply Theorem~\ref{9} and we conclude that
\begin{equation}
\text{$\on{Im}(\rho_*(\hN) \to j'_*\hL)\simeq\rho_*(\on{Im} \hN\to j_*\cM_2)$
is a coherent $\hcA_{X'}$-module.}
\end{equation}
Now we can again apply Proposition~\ref{coherence from finite morphisms},
and conclude that $\on{Im}(\hN\to j_*\hM_2)$
is a coherent  $\hE_X(0)\vert_U$-module.
Therefore $\on{Im}(\hM\to j_*\hM_2)$
is a coherent  $(\hE_X\vert_U)$-module.

\subsection{The codimension three extension theorem}
\label{convergent}

In this subsection we prove Theorem~\ref{1}
by deducing it from the formal version Theorem~\ref{3} which was proved in subsection~\ref{Formal} and from the comparison  Theorem~\ref{8}
whose proof is given later in section \ref{comparison}.

We work in the geometric setting of subsection~\ref{General}
with $U$ an open subset of $T^*X$, $\Lambda$ a closed Lagrangian
analytic subset of $U$, and $Y$ a closed analytic subset of $\Lambda$
of codimension  at least three. We are given a holonomic  $\bl\cE_X\vert_{U\setminus Y}\br$-module $\cM$
whose support is contained in $\Lambda\setminus Y$ and an
 $\bl\cE_X(0)\vert_{U\setminus Y}\br$-lattice $\cN$ of $\cM$. We write $j\cl U\setminus Y\to U$ for the open inclusion and we are to show that $j_*\cM$ is a coherent $(\cE_X|_U)$-module.

By replacing $ \cN$ with
$\cext_{\cE_X(0)\vert_{U\setminus Y}}^n
\bl\cext_{\cE_X(0)\vert_{U\setminus Y}}^n(\cN,\cE_X(0)\vert_{U\setminus Y}),
\cE_X(0)\vert_{U\setminus Y}\br$, we may assume from the beginning that
\eq
&&\cN\simeq \cext_{\cE_X(0)\vert_{U\setminus Y}}^n
\bl\cext_{\cE_X(0)\vert_{U\setminus Y}}^n(\cN,\cE_X(0)\vert_{U\setminus Y}),
\cE_X(0)\vert_{U\setminus Y}\br.\label{eq:Nref}
\eneq

  We first make use of Propositions~\ref{coherence from finite morphisms}
 and  ~\ref{vector bundle} to conclude that $\drho_*(\cM\vert_{\dU})$ is  a locally free $\cK_{X'\setminus Y'}$-module
of finite rank and $\drho_*(\cN\vert_{\dU})$
is a coherent $\cA_{X'\setminus Y'}$-lattice of $\drho_*(\cM\vert_{\dU})$.

We pass to the formal setting $\hM \seteq (\hE_X\vert_{U\setminus Y})\otimes _{\cE_X\vert_{U\setminus Y}}\cM$ and $\hN \seteq (\hE_X\vert_{U\setminus Y})\otimes _{\cE_X\vert_{U\setminus Y}}\cN$.
We now apply Theorem~\ref{3} which was proved in subsection~\ref{Formal}  to conclude that
$j_*\hM$ is a coherent $(\hE_X|_U)$-module. We again make use of
Proposition~\ref{vector bundle} to conclude that $\rho_*(j_*\hM)$ is a
locally free $\hcK_{X'}$-module of finite rank.
As we work locally near a point $y\in Y$, we may then assume that
$j'_*\drho_*\hM\simeq\rho_*j_*\hM$ is a free $\hcK_{X'}$-module.
Hence $j'_*\drho_*\hM$ has a free $\hcA_{X'}$-lattice
$\hL$.

Now we will employ the following comparison lemma between convergent lattices and formal lattices.
\Lemma
\label{lattice comparison}
Let $Z$ be a complex manifold and let $\cM$ be a coherent $\cK_Z$-module,
and $\hM\seteq\hK_Z\otimes_{\cK_Z}\cM$.
Then the set $\bbL(\cM)$ of $\cA_Z$-lattices of $\cM$
and  the set $\bbL(\hM)$ of $\hcA_Z$-lattices of $\hM$
are in one to one  correspondence:  the lattices $\cN\in\bbL(\cM)$
and  $\hN\in\bbL(\hM)$ correspond to each other via
$$\hN=\hcA_Z\otimes_{\cA_Z}\cN \qquad
\cN=\cM\cap\hN\,.
$$
Moreover we have an isomorphism
$\cN/\cN(-1)\isoto\hN/\hN(-1)$.
\enlemma
As the proof is by a routine argument we omit it.

\medskip
Thus, we have the following Cartesian square:
\begin{equation}
\begin{gathered}
\xymatrix{
 {\drho_*\cM}\ar@{_{(}->}[d] \ar@{}[dr]|{\square}&
{\strut\;\drho_*\cM\cap  {j'}^{-1}\hL}
\ar@{_{(}->}[l]\ar@{_{(}->}[d]\\
 {{j'}^{-1}\rho_*\hM }
& {{j'}^{-1}\hL}\ar@{_{(}->}[l]\,.
}
\end{gathered}
\end{equation}
Therefore $\cL\seteq\drho_*\cM\cap\hL$ is an $\cA_{X'\setminus Y'}$-lattice of
$\drho_*\cM$ by the lemma above.
Since $\cL/\cL(-1)\simeq\hL/\hL(-1)$ is a free $\cO_{X'\setminus Y'}$-module,
$\cL$ is a locally free $\cA_{X'\setminus Y'}$-module.

We now apply Theorem~\ref{8} to the lattice $\cL$ and we then conclude that
$j'_*\cL$ is a locally free $\cA_{X'}$-module.

On the other hand, by shrinking $X'$ if necessary,
there exist integers $p, q$ such that
$\cL\subset t^p\drho_*\cN\subset t^q\cL$.
Hence for any $s\in j'_*\cL$,
$\DA_{X'}s\subset t^pj'_*\drho_*\cN\subset j'_*t^q\cL$.
Hence $\DA_{X'}j'_*\cL\subset j'_*t^q\cL$.
Since $\DA_{X'}j'_*\cL$ is a sum of coherent $\cA_{X'}$-submodules
of the coherent $\cA_{X'}$-module $j'_*t^q\cL$,
we conclude that
$\DA_{X'}j'_*\cL$ is a coherent $\cA_{X'}$-module.
Now we make use of Proposition~\ref{coherence from finite morphisms} and conclude that $j_*\cM$ is a coherent $\cE_X\vert_U$-module.

\subsection{The codimension two submodule extension theorem}

In this subsection we prove Theorem~\ref{2}
by deducing it from the formal version Theorem~\ref{4} which was proved in subsection~\ref{formal submodule theorem proof} and making use of  the comparison  Theorem~\ref{8}
whose proof is given later in section \ref{comparison}.

We proceed as in subsection~\ref{convergent} this time making use of the fact that we have already proved the formal version Theorem~\ref{4}.
We write $\hM \seteq \hE_X\otimes _{\cE_X}\cM$ and $\hM_1 \seteq \hE_X\otimes _{\cE_X}\cM_1$. By  Theorem~\ref{4} we conclude that $\hM_1$ extends uniquely to a holonomic $\hE_X$-module $j_*\hM_1$ on $U$. As in the previous section we argue that  $\rho_*j_*\hM_1$ is a locally free $\hcK_{X'}$-module.
As we work locally near a point $y\in Y$,
we may then assume that
$j'_*\drho_*\hM_1\simeq\rho_*j_*\hM_1$ is a free  $\hcK_{X'}$-module.
Hence $j'_*\drho_*\hM_1$ has a free $\hcA_{X'}$-lattice $\hL$.

 Making use of Lemma~\ref{lattice comparison} again,
we conclude as above that  $\cL\seteq \drho_*\cM_1\cap {j'}^{-1}\hL$  is an $\cA_{X'\setminus Y'}$-lattice of  $\drho_*\cM_1$.
Since $\cL/\cL(-1)\simeq{j'}^{-1}\bl\hL/\hL(-1)\br$ is a free $\cO_{X'\setminus Y'}$-module,
$\cL$ is a locally free $\cA_{X'\setminus Y'}$-module. As in the previous section we apply Theorem~\ref{8} to the lattice $\cL$ and we then conclude that
$j'_*\cL$ is a locally free $\cA_{X'}$-module.
\ Again as in the previous section,
$\DA_{X'}j'_*\cL$ is coherent over $\cA_{X'}$.
Now we make use of Proposition~\ref{coherence from finite morphisms} and conclude that $j_*\cM_1$ is coherent.

\section{The commutative formal codimension three extension theorem}
\label{The formal case}

This section is devoted to the proof of Theorem~\ref{5}. Let us recall
our setup.
We consider a complex manifold $X$  and a subvariety $Y$ of $X$
such that the codimension of $Y$ in $X$ is at least 3.
 We write $j\cl X\setminus Y \to X$ for the open inclusion.
We are given  a reflexive coherent $\hcA_{X\setminus Y}$ module $\hN$ on
$X\setminus Y$.
We will show that $j_*\hN$ is a coherent $\hcA_X$-module.
As the question is local, proceeding inductively along $Y$, we can and we will assume that $Y$ is a smooth submanifold
without loss of generality.

Recall that we write $\odh{\hN}$ for the dual of $\hN$, i.e.,
\begin{equation}
\odh{\hN} \ =  \ \chom_{\str_{X\setminus Y}}(\hN,\str_{X\setminus Y})\,.
\end{equation}
Since $\hN$ is reflexive, the homomorphism
$\hN\to\odh{\odh{\hN}}$ is an isomorphism.
We shall show that
$j_*\hN$ is a coherent $\str_X$-module.

The idea of the proof is as follows. As $\hN$ is a coherent
$\str_{X\setminus Y}$-module, we have
\begin{equation}
\hN\isoto\varprojlim\; \hN/t^k\hN
\end{equation}
where the $\hN/t^k\hN$ are coherent
$\cO_{X\setminus Y}\otimes(\C[t]/t^k\C[t])$-modules
and of course they are also coherent as $\cO_{X\setminus Y}$-modules
(see \cite{KS3}).
We  write $\hN_k= \hN/t^{k}\hN$ and by convention we set $\hN_k=0$ for $k<0$.
However, there is no reason to expect that the $\hN_k$ are reflexive
as $\cO_{X\setminus Y}$-modules.
We can remedy this situation by replacing them with $(\hN_k)^{**}$;
here the dual is taken in the sense of $\cO_{X\setminus Y}$-modules,
i.e., $\hN_k^*=\chom_{\cO_{X\setminus Y}}(\hN_k, \cO_{X\setminus Y})$. As the sheaves $(\hN_k)^{**}$ are reflexive,
we can make use of the classical extension Theorem~\ref{6}  and so we know that the $j_*(\hN_k)^{**}$ are coherent
$\cO_X$-modules. We will then show that
\begin{subequations}
\begin{equation}
\label{minor point}
\hN\isoto\varprojlim\; (\hN_k)^{**}
\end{equation}
and
\begin{equation}
\label{main point}
\text{$\varprojlim j_*(\hN_k)^{**}$ is a coherent $\str_X$-module.}
\end{equation}
\end{subequations}

The key, of course, is to show \eqref{main point}.

In what follows we will be making use of the following well-known characterization of torsion free and reflexive sheaves.
\Lemma
\label{reflexive and torsion free}
Let $Z$ be a smooth complex manifold.
\bnum\item
Let $\cF$ be a coherent $\cO_Z$-module.
Then we have
\be[{\rm(a)}]
\item
$\cF$ is torsion free
if and only if $\codim_Z(\Supp\;\cext^i_{\cO_Z}(\cF,\cO_Z))\geq i+1$ for any $i>0$.
\item
$\cF$ is reflexive
if and only if $\codim_Z(\Supp\;\cext^i_{\cO_Z}(\cF,\cO_Z))\geq i+2$
for any $i>0$.
\ee
\item
Let $\hF$ be a coherent $\hcA_Z$-module.
Then we have
\be[{\rm(a)}]
\item
$\hF$ is torsion free \ro i.e., $\hF\to\odh\odh(\cF)$
is a monomorphism\rf\
if and only if
$\codim_Z\bl\Supp\;\cext^i_{\hcA_Z}(\hF,\hcA_Z)\br\geq i$ and
$\codim_Z\bl\Supp\;\cext^i_{\hcA_Z}(\hF,\hcK_Z)\br\geq i+1$ for any $i>0$.
\item
$\hF$ is reflexive
if and only if
$\codim_Z\bl\Supp\;\cext^i_{\hcA_Z}(\hF,\hcA_Z)\br\geq i+1$ and
$\codim_Z\bl\Supp\;\cext^i_{\hcA_Z}(\hF,\hcK_Z)\br\geq i+2$ for any $i>0$.
\label{reflexive characterization}
\ee
\ee
\enlemma

We comment briefly on the proof. The criterion (i) is well-known and can be found in \cite[Chapter 1]{ST}, for example. As for (ii), it is a statement on the level of local rings and the dimension of the local ring $\hcA_{Z,x}$ is $\dim(Z) +1 $ (we have added one formal dimension to $Z$). Thus, we conclude (ii).

Note that
we have
\eq
&&\codim_{\mathrm{Spec}(\hcA_Z)}\bl\Supp\;\cext^i_{\hcA_Z}(\hF,\hcA_Z)\br\\
&&\hs{5ex}=\sup\bl 1+\codim_Z\Supp\;\cext^i_{\hcA_Z}(\hF,\hcA_Z),
\codim_Z\Supp\;\cext^i_{\hcA_Z}(\hF,\hcK_Z)\br.\nonumber
\eneq

We also make use of the following:

\Lemma\label{residue}
Let $Z$ be a smooth complex manifold,
and let $\hF$ be a coherent $\hcA_Z$-module
which is $t$-torsion.
Then we have
$$\cext^1_{\str_{Z}}(\hF, \str_{Z})
\simeq\chom_{\cO_{Z}}(\hF, \cO_{Z}) = (\hF)^*\,.$$
\enlemma
\Proof
As the sheaf $\hF$ is $t$-torsion and
$t$ acts bijectively on $\hcK_X$, we conclude that
\begin{equation*}
\cext^k_{\hcA_X}(\hF,\hcK_X) \ = \ 0 \quad \text{for all $k$.}
\end{equation*}
Making use of the long exact sequence associated to the exact sequence
\begin{equation*}
0 \to \hcA_{Z} \to \hcK_{Z} \to \hcK_{Z}/\hcA_{Z} \to 0\,,
\end{equation*}
we conclude that
\begin{equation}
\cext^1_{\hcA_{Z}}(\hF,\hcA_{Z}) \simeq  \chom_{\hcA_{Z}}(\hF, \hcK_{Z}/\hcA_{Z})\,.
\end{equation}
Now we have
\begin{equation}
 \hcK_{Z}/\hcA_{Z}  \simeq \cO_{Z}t^{-1} \oplus  \cO_{Z}t^{-2} \oplus \cdots\
\quad\text{as an $\cO_Z$-module.}
\end{equation}
Thus, an element $f\in \chom_{\hcA_{Z}}(\hF, \hcK_{Z}/\hcA_{Z})$
consists of $f_{k} \in  \chom_{\cO_{Z}}(\hF, \cO_{Z}t^{-k})$ ($k\ge1$) such that
$f_{k}\circ t = t \circ f_{k+1}$.
By mapping $f\mapsto f_{1}$ we get an isomorphism
\begin{equation*}
 \chom_{\hcA_{Z}}(\hF, \hcK_{Z}/\hcA_{Z}) \simeq
\chom_{\cO_{Z}}(\hF, \cO_{Z}) \,.
\end{equation*}
This gives us the desired result.
\QED

We will now begin the proof of \eqref{minor point} and \eqref{main point}. To that end
we consider the exact sequence
\begin{equation*}
0 \To \hN \To[{\ t^k\ }]\hN \To \hN_k \to 0\,.
\end{equation*}
Dualizing it we obtain:
\begin{multline*}
0 \to \odh{\hN} \xrightarrow {t^k}\odh{\hN}
\to \cext^1_{\str_{X\setminus Y}}( \hN_k,\hcA_{X\setminus Y}) \\
\to\cext^1_{\str_{X\setminus Y}}( \hN,\hcA_{X\setminus Y})
  \xrightarrow {t^k}\cext^1_{\str_{X\setminus Y}}( \hN,\hcA_{X\setminus Y}).
\end{multline*}

By making use of Lemma~\ref{residue} we can rewrite the exact sequence as follows:
\begin{multline}
\label{auxiliary exact sequence}
0 \to (\odh{\hN})_k \to (\hN_k)^*
\\
\to  \Ker\bl\cext^1_{\str_{X\setminus Y}}( \hN,\hcA_{X\setminus Y})
\xrightarrow {t^k}\cext^1_{\str_{X\setminus Y}}( \hN,\hcA_{X\setminus Y})\br
\to 0,
\end{multline}
where we have written $(\odh{\hN})_k\seteq\odh{\hN}/t^k\odh{\hN}$.
Setting
\begin{equation*}
\cG_k \seteq \Ker(\cext^1_{\str_{X\setminus Y}}( \hN,\hcA_{X\setminus Y})  \xrightarrow {t^k}\cext^1_{\str_{X\setminus Y}}( \hN,\hcA_{X\setminus Y})),
\end{equation*}
we rewrite this sequence as
\begin{equation}
\label{first basic exact sequence}
0 \to (\odh{\hN})_k \to (\hN_k)^* \to \cG_k\to 0.
\end{equation}
As $\hN$ is reflexive and  making use of Lemma~\ref{reflexive and torsion free} \eqref{reflexive characterization} we conclude that
\begin{equation*}
\codim_{X\setminus Y}\Supp\bl \cext^1_{\str_{X\setminus Y}}( \hN,\hcA_{X\setminus Y})\br
\geq 2 \,.
\end{equation*}
Thus, by definition of $\cG_k$ we obtain
\begin{equation}
\codim_{X\setminus Y}(\Supp (\cG_k))\geq 2 \,.
\end{equation}
This, in turn, implies that
\begin{equation*}
\chom_{\cO_{X\setminus Y}}(\cG_k,\cO_{X\setminus Y})
=\cext^1_{\cO_{X\setminus Y}}(\cG_k,\cO_{X\setminus Y})=0\,.
\end{equation*}

 Finally, dualizing the exact sequence \eqref{first basic exact sequence} we conclude that
 \begin{equation}
 \label{bidual}
( \hN_k)^{**} \isoto (\odh{\hN})_k^*\,.
 \end{equation}

Substituting $\odh{\hN}$ for $\hN$ in  \eqref{auxiliary exact sequence}  and setting
\begin{equation}\label{def:Fk}
\cF_k \seteq \Ker(\cext^1_{\str_{X\setminus Y}}( \odh{\hN},\hcA_{X\setminus Y})  \xrightarrow {t^k}\cext^1_{\str_{X\setminus Y}}( \odh{\hN},\hcA_{X\setminus Y})),
\end{equation}
we obtain the exact sequence
\begin{equation*}
0 \to \hN_k \to (\odh{\hN})_k^* \to \cF_k\to 0.
\end{equation*}
Making use of \eqref{bidual} we can rewrite this exact sequence in the following form:
\begin{equation}
\label{second basic exact sequence}
0 \to \hN_k \to( \hN_k)^{**} \to \cF_k\to 0.
\end{equation}
Arguing as we did before for $\cG_k$ we see that
\begin{equation}
\label{support of F}
\codim_{X\setminus Y}(\Supp (\cF_k))\geq 2 \,.
\end{equation}
The last two statements form the basis for the rest of the argument.
The $\cF_k$ form an increasing sequence of coherent submodules of
a coherent $\hcA_{X\setminus Y}$-module
$\cext^1_{\str_{X\setminus Y}}( \odh{\hN},\hcA_{X\setminus Y})$.
Hence the union $\cF\seteq\cup_k\cF_k$ is a coherent $\hcA_{X\setminus Y}$-module.
Note that by definition, $\cF$ is precisely the  $t$-torsion  part
$\cext^1_{\str_{X\setminus Y}}( \odh{\hN},\hcA_{X\setminus Y})_{t-\mathrm{tors}}$ of
$\cext^1_{\str_{X\setminus Y}}( \odh{\hN},\hcA_{X\setminus Y})$.
Hence $t$ acts locally nilpotently on $\cF$, and
$\cF$ is a coherent $\cO_{X\setminus Y}$-module.

Let us us introduce the $\hcA_X$-module $\tN_k$ by setting
\begin{equation*}
\tN_k \ = \ j_*\bl(\hN_k)^{**}\br\,.
\end{equation*}
By the classical Theorem~\ref{6},
the $\tN_k$ are coherent $\cO_X$-modules.

Let us write $i_k\cl \tN_k \to \tN_{k+1}$ for the map induced by the  multiplication map $t\cl\hN_k \to \hN_{k+1}$ and $p_k \cl \tN_{k+1} \to \tN_{k}$ for the map induced by the natural projection $\hN_{k+1}\to \hN_k$. Then we have a commutative diagram
\begin{equation*}
\xymatrix@C=6ex{
\tN_k\ar[r]^-{i_k}\ar[d]^-{p_{ k-1}}\ar[dr]|-{t}&\tN_{k+1}\ar[d]^-{p_{ k}}\\
\tN_{ k-1}\ar[r]^-{i_{k-1}}&\tN_{k}\,.
}
\end{equation*}

Recall that we already know that the $\tN_k$ are coherent $\cO_X$-modules.
Thus,  in order to prove that  $\varprojlim \tN_k $ is coherent
it suffices to show, according to \cite[Proposition 1.2.18]{KS3}, that
the pro-objects (see ibid.) $\proolim[k]\ker(\tN_k \xrightarrow{t}\tN_k)$
and  $\proolim[k]\coker(\tN_k \xrightarrow{t}\tN_k)$
are locally represented by coherent $\hcA_{X}$-modules.
These pro-objects are isomorphic to
 $\proolim[k]\ker(\tN_k \To[{i_k}]\tN_{k+1})$
and   $\proolim[k]\coker(\tN_k \To[{i_k}]\tN_{k+1})$,
respectively.
Hence  $\varprojlim \tN_k $ is coherent
as soon as\begin{subequations}
\begin{multline}
\label{stability of ker}
\parbox{60ex}{$\ker(\tN_k \To[{i_k}]\tN_{k+1})\to\ker(\tN_{k-1} \To[{i_{k-1}}]\tN_{k})$
is an isomorphism locally for $k\gg0$,}
\end{multline}
\begin{multline}
\label{stability of coker}
\parbox{65ex}{$\coker(\tN_k \To[{i_k}]\tN_{k+1})\to\coker(\tN_{k-1} \To[{i_{k-1}}]\tN_{k})$
is an isomorphism locally for $k\gg0$.}
\end{multline}
\end{subequations}
In order to prove the statements above
we  consider the following commutative diagram:
\begin{equation}
\label{first}
\begin{gathered}
\xymatrix@C=7ex{
0\ar[r]&{\odh{\hN}} \ar[r]^{t^{k+1}}&{\odh{\hN}}\ar[r]&{(\odh{\hN})}_{k+1}
\ar[r]&0\\
0\ar[r]&{\odh{\hN}} \ar[r]^{t^{k}}\ar[u]_{\id}&{\odh{\hN}}\ar[r]\ar[u]_{t}
&{(\odh{\hN})}_k\ar[r]\ar[u]_t&0}
\end{gathered}
\end{equation}
with exact rows.

Dualizing the diagram,
we obtain a commutative diagram
\begin{equation*}
\xymatrix@C=2.7ex{
(\hN_{k+1})^{**}\ar[d]^-{\bwr}\\
\cext^1_{\hcA_{X\setminus Y}}((\odh{\hN})_{k+1},\hcA_{X\setminus Y})\ar[d]
\ar[r]&\cext^1_{\hcA_{X\setminus Y}}(\odh{\hN},\hcA_{X\setminus Y})
\ar[r]^{t^{k+1}}\ar[d]^t&
\cext^1_{\hcA_{X\setminus Y}}(\odh{\hN},\hcA_{X\setminus Y})\ar[d]_{\id}\\
\cext^1_{\hcA_{X\setminus Y}}((\odh{\hN})_{k},\hcA_{X\setminus Y})
\ar[r]&\cext^1_{\hcA_{X\setminus Y}}(\odh{\hN},\hcA_{X\setminus Y})\ar[r]^{t^k}&
\cext^1_{\hcA_{X\setminus Y}}(\odh{\hN},\hcA_{X\setminus Y})\\
(\hN_{k})^{**}\ar[u]^-{\bwr}}\label{dia:aa}
\end{equation*}

Making use of \eqref{second basic exact sequence}
and \eqref{def:Fk}, we obtain the commutative diagram:
\begin{equation}
\begin{gathered}
\label{p diagram}
\xymatrix{
0 \ar[r]&\hN_{k+1}\ar[r]\ar[d]&(\hN_{k+1})^{**}\ar[d]^{p_{k}\vert_{X\setminus Y}}\ar[r]
&\cF_{k+1}\ar[d]^{t}
\ar[r]& 0\\
0\ar[r]&\hN_{k}\ar[r]&(\hN_{k})^{**}\ar[r]&\cF_{k}\ar[r]&0
}
\end{gathered}
\end{equation}
with exact rows. Here the commutativity of the right
square follows from the previous diagram.

As the first column is a surjection, we conclude:
\eq
&&\coker(p_{k})|_{X\setminus Y}\ \simeq\coker(\cF_{k+1}\To[\ t\ ]\cF_{k})\,.
\label{eq:cok fk}
\eneq
Hence along with \eqref{support of F}, we obtain the estimate
\eq\label{cokernel of p}
&&\codim_{X}(\Supp (\coker(p_k)))\geq 2 \,.
\eneq

We make one further observation at this point. {}From  \eqref{p diagram} we also obtain an exact sequence
\begin{equation*}
0 \to \hN \to \varprojlim (\hN_k)^{**} \to \varprojlim \cF_k\,,
\end{equation*}
where the last projective system is given by
\begin{equation*}
\dots \to \cF_{k+1} \xrightarrow{t} \cF_k \xrightarrow{t} \cF_{k-1} \to \dots\;.
\end{equation*}
We now recall that
 \begin{equation*}
\cF_k \subset \cF\seteq
\cext^1_{\str_{X\setminus Y}}(\odh{\hN},\hcA_{X\setminus Y})_{t-\mathrm{tors}}\,.
\end{equation*}
Since $\cF$ is a coherent $\cO_{X\setminus Y}$-module, we obtain
\begin{equation}
\label{local stabilization}
\begin{gathered}
\text{Locally on $X\setminus Y$ there exists a integer $k_0$ such that}
\\
 \text{$t^{k_0}\cF_k=0$ for all $k$.}
\end{gathered}
\end{equation}

This implies that

\begin{equation}
\varprojlim \cF_k \ = \ 0\quad\text{on $X\setminus Y$.}
\end{equation}

Thus, we obtain \eqref{minor point}. The remainder of this section
is devoted to the proof of \eqref{main point}.

Let us consider a slight variant of \eqref{p diagram}:
\eqn
\xymatrix{
0 \ar[r]&\hN_k\ar[r]\ar[d]^t&(\hN_k)^{**}\ar[r]\ar[d]^t&\cF_k \rule[-1.5ex]{0ex}{2ex}\ar[r]\ar@{>->}[d]&0\\
0 \ar[r]&\hN_{k+1}\ar[r]&(\hN_{k+1})^{**}\ar[r]&\cF_{k+1} \ar[r]&0
}
\eneqn
where the rows are exact.

As $\hN$ is torsion free, the left vertical arrow is a monomorphism.
As the right vertical arrow is an inclusion, we conclude that $(\hN_k)^{**}  \xrightarrow{t} (\hN_{k+1})^{**} $ is also a monomorphism and hence

\begin{equation}
\text{The maps $i_k\cl \tN_k \to \tN_{k+1}$ are monomorphisms}\,.
\end{equation}
It of course implies \eqref{stability of ker}.
It only remains to prove
\eqref{stability of coker}.

We will argue next that the square
\begin{equation}
\label{Cartesian}
\begin{CD}
\tN_k @>{i_k}>> \tN_{k+1}
\\
@V{p_{k-1}}VV @VV{p_k}V
\\
\tN_{k-1} @>{i_{k-1}}>> \tN_k
\end{CD}
\end{equation}
is Cartesian.

To this end we consider the following commutative diagram:

\begin{equation*}
\begin{CD}
@. 0 @. 0 @. 0 @.
\\
@. @VVV @VVV @VVV @.
\\
0 @>>> \hN_k @>{(p_{k-1},i_{k})}>> \hN_{k-1}\oplus\hN_{k+1}@>{(i_{k-1},-p_k)}>> \hN_k
\\
@. @VVV @VVV @VVV @.
\\
0 @>>> (\hN_k)^{**} @>{(p_{k-1},i_{k})}>> (\hN_{k-1})^{**}\oplus(\hN_{k+1})^{**}@>{(i_{k-1},-p_k)}>> (\hN_k)^{**}
\\
@. @VVV @VVV @VVV @.
\\
0 @>>> \cF_k @>{(t,\id)}>> \cF_{k-1}\oplus\cF_{k+1}@>{(\id,-t)}>> \cF_k
\\
@. @VVV @VVV @VVV @.
\\
@. 0 @. 0 @. 0 @.
\end{CD}
\end{equation*}
The columns are exact as they are obtained from \eqref{second basic exact sequence} and the top row is exact as
\begin{equation*}
\begin{CD}
\hN_k @>{i_k}>> \hN_{k+1}
\\
@V{p_{k-1}}VV @VV{p_k}V
\\
\hN_{k-1} @>{i_{k-1}}>>\hN_k
\end{CD}
\end{equation*}
is Cartesian. We  check easily that the bottom row is exact as follows. Let $(a,b)\in \cF_{k-1}\oplus\cF_{k+1}$ such that $a=tb$. Then, clearly, $b\in\cF_k $. Thus, we conclude that  the middle row is also exact and hence
\begin{equation*}
\begin{CD}
(\hN_k)^{**} @>{i_k}>> (\hN_{k+1})^{**}
\\
@V{p_{k-1}}VV @VV{p_k}V
\\
(\hN_{k-1})^{**}@>{i_{k-1}}>> (\hN_k)^{**}
\end{CD}
\end{equation*}
is Cartesian. As $j_*$ is a left exact functor we conclude that \eqref{Cartesian} is Cartesian.

Let us now consider the following commutative diagram
\eq
&&\hs{.9ex}\xymatrix@C=2.2ex{
&0\ar[d]& 0\ar[d]\\
0\ar[r]&\ker(p_{k-1})\ar[r]\ar[d]&\ker(p_k)\ar[r]\ar[d]&0\ar[d]\\
0\ar[r]&\tN_k \ar[d]^{p_{k-1}}\ar[r]^{i_k}&\tN_{k+1}\ar[r]\ar[d]^{p_k}&\coker(i_k)
\ar[r]\ar[d]&0\\
0\ar[r]&\tN_{k-1}\ar[d]\ar[r]^{i_{k-1}}&\tN_k\ar[r]\ar[d]&\coker(i_{k-1})
 \ar[r]\ar[d]&0\\
0\ar[r]&\coker(p_{k-1})\ar[d]\ar[r]&\coker(p_k)\ar[d]\ar[r]&
\coker(\tN_{k-1}\oplus\tN_{k+1}\xrightarrow{}\tN_k)\ar[d]\ar[r]&0\\
&0&0&0
}\label{diag:square}
\eneq
The fact that the square \eqref{Cartesian}  is Cartesian implies
\begin{equation}
\parbox{64ex}
{$\coker(i_k) \to  \coker(i_{k-1}) $ and $\coker(p_{k-1}) \to \coker(p_k)$
are monomorphisms.}
\end{equation}
Thus all the rows and columns in the diagram above are exact.

By induction on $\codim\; Y$, it suffices to prove
\eqref{stability of coker} in the neighborhood of a smooth point $y$ of $Y$.

We will now analyze $\coker(\tN_{k-1}\oplus\tN_{k+1} \xrightarrow{}\tN_k)$.
{}From \eqref{cokernel of p} it follows that
\begin{equation}
\label{weak estimate}
\codim_X\Supp(\coker(\tN_{k-1}\oplus\tN_{k+1} \xrightarrow{}\tN_k)) \geq 2.
\end{equation}
We will next argue:
\begin{equation}
\ba{l}
\codim_X\Supp(\coker(\tN_{k-1}\oplus\tN_{k+1} \xrightarrow{}\tN_k)) \geq 3
\quad\text{for $k\gg0$}\\[1ex]
\hs{40ex}\text{on a neighborhood of $y$ .}\ea
\end{equation}

To do so, we  write $X$ locally as a product $X\simeq D_m \times D_\ell$,
where $D_m$ is an $m$-dimensional  ball of radius 2 in $\bC^m$
and $Y$ corresponds to $\{0\}\times D_\ell$ and $y$ to $(0,0)$.
Hence by the assumption, we have $m\ge3$.
Let us write $\pi\cl X\to D_\ell$ for the projection with respect to
this decomposition.
Let us take a relatively compact neighborhood $U$ of $0\in D_\ell$,
and set $W'\seteq\set{z\in D_m}{1/2<|z|<1}$ and $W\seteq
\set{z\in D_m}{|z|<1}$.
By compactness, the $\cF_k$ stabilize on $W'\times U$:
\begin{equation}
\begin{gathered}
\text{there exists an integer $k_0$ such that for $k\geq k_0$}
\\
 \text{$\cF_k\vert_{W'\times U}=\cF_{k_0}\vert_{W'\times U}$.}
\end{gathered}
\end{equation}
By \eqref{eq:cok fk},
$\coker(p_{k})|_{X\setminus Y}\simeq \cF_{k}/t\cF_{k+1}$ and therefore
$$\coker(p_{k-1})\vert_{W'\times U}\to\coker(p_k)\vert_{W'\times U}$$
is an isomorphism for $k\geq k_0$.
Thus, the exactness of the bottom row in \eqref{diag:square} implies
\begin{equation*}
\Supp(\coker(\tN_{k-1}\oplus\tN_{k+1} \xrightarrow{}\tN_k))\vert_{W'\times U}
=\emptyset \qquad \text{for $k\geq k_0$.}
\end{equation*}
Therefore,  the projection $\pi$ restricted to
$$\Supp(\coker(\tN_{k-1}\oplus\tN_{k+1} \xrightarrow{}\tN_k))
\cap (W\times U)\To U$$
is proper and hence it is a finite morphism.
This implies that
$$\codim_X\Supp(\coker(\tN_{k-1}\oplus\tN_{k+1} \xrightarrow{}\tN_k))\cap (W\times U)
\geq \dim W\ge3$$
for $k\ge k_0$.

Thus, shrinking $X$ if necessary, we may assume that
\begin{equation}
\label{strong estimate}
\text{$\codim_X\Supp(\coker(\tN_{k-1}\oplus\tN_{k+1} \xrightarrow{}\tN_k))
\geq 3$ for $k\gg0$.}
\end{equation}

Let us recall the Cartesian square \eqref{Cartesian}:

\begin{equation}
\xymatrix{
\tN_k\ar[d]_{p_{k-1}} \ar[r]^{i_k}\ar@{}[dr]|{\square}&\tN_{k+1}\ar[d]^ {p_k}\\
\tN_{k-1}\ar[r]^{i_{k-1}}&\tN_k\,.}
\label{Cartesian1}
\end{equation}
We have seen that this Cartesian square has the following properties:
\eq
&&\parbox{65ex}{
\bnum
\item the $\tN_k$ are reflexive $\cO_X$-modules,\label{sq1}
\item $\tN_k=0$ for $k\le 0$.\label{sq0}
\item the $i_k$ are monomorphisms,\label{sq2}
\item $\codim_X \Supp(\coker(p_k))\ge2$,\label{sq3}
\item $\codim_X\Supp\bl\coker(\tN_{k-1}\oplus\tN_{k+1}\To\tN_k)\br\ge 3$
for $k\gg0$.
\label{sq4}
\ee
}\label{cond:square}
\eneq

We dualize this square to obtain:

\begin{equation}
\label{Dual Cartesian}
\begin{CD}
\tN_k^*@>{p_k^*}>> \tN_{k+1}^*
\\
@V{i_{k-1}^*}VV @VV{i_k^*}V
\\
\tN_{k-1}^* @>{p_{k-1}^*}>> \tN_k^*.
\end{CD}
\end{equation}

Assuming only \eqref{cond:square},
we shall show that this square is Cartesian and it also satisfies
the properties in \eqref{cond:square}.
Of course, \eqref{sq1} and \eqref{sq0} are obvious.

Because  $\codim_X(\Supp(\coker(p_k)))\geq 2$, we have:
\begin{equation}
\text{the maps $p_k^*$ are monomorphisms.}
\end{equation}

As \eqref{Cartesian1} is Cartesian, the sequence

\begin{equation*}
\begin{CD}
0 @>>> \tN_k @>{(p_{k-1},i_{k})}>> \tN_{k-1}\oplus\tN_{k+1}@>{(i_{k-1},-p_k)}>> \tN_k
\end{CD}
\end{equation*}
is exact.
Let us brake this into two exact sequences:

\begin{equation}
\begin{CD}
0 @>>> \tN_k @>{(p_{k-1},i_{k})}>> \tN_{k-1}\oplus\tN_{k+1}@>>> \cK @>>> 0
\end{CD}
\end{equation}
and
\begin{equation}
\label{eq:56}
0 \to \cK \to  \tN_k \to  \tN_k/\cK \to 0.
\end{equation}
Dualizing the first exact sequence we obtain an exact sequence
\begin{equation}
\begin{CD}
0 @>>> \cK^* @>>> \tN_{k-1}^*\oplus\tN_{k+1}^*@>{(p_{k-1}^*, i_k^*)}>>
\tN_k ^*\,.
\end{CD}\label{eq:kex}
\end{equation}

By \eqref{cond:square} \eqref{sq4},
we conclude that
$\codim_X(\Supp(\tN_k/\cK))\geq 3$. Hence \eqref{eq:56} implies that
\begin{equation*}
 \tN^*_k = \cK^*\,.
\end{equation*}

The exact sequence \eqref{eq:kex} then reads as
\begin{equation*}
\begin{CD}
0 @>>> \tN_k^* @>{(i_{k-1}^*,-p_{k}^*)}>> \tN_{k-1}^*\oplus\tN_{k+1}^*@>{(p_{k-1}^*,i_k^*)}>> \tN_k ^*
\end{CD}
\end{equation*}
and thus we conclude that  \eqref{Dual Cartesian} is Cartesian.

Let us now show that the $i_k^*$ satisfy \eqref{sq3}.

Dualizing
\begin{equation*}
\begin{CD}
0 @>>> \tN_k @>{i_k}>> \tN_{k+1} @>>> \coker(i_k)@>>> 0
\\
@. @V{p_{k-1}}VV @VV{p_k}V @VVV @.
\\
0 @>>> \tN_{k-1} @>{i_{k-1}}>> \tN_k @>>> \coker(i_{k-1}) @>>> 0,
\end{CD}
\end{equation*}
we obtain
\begin{equation}
\label{auxiliary for duals of i_k}
\xymatrix{
\tN_{k+1}^* \ar[r]^{i_k^*}&\tN^*_k\ar[r]&\cext^1(\coker(i_k),\cO_X)\\
\tN^*_k \ar[r]^{i_{k-1}^*}\ar[u]_{p_{k-1}^*}&\tN_{k-1}\ar[r]\ar[u]^{p_k^*}
&\cext^1(\coker(i_{k-1}),\cO_X).\ar[u]
}
\end{equation}

Let us now make a sequence of observations. First,  $\coker(i_k)$ is torsion free as   $\coker(i_k)\subset \coker(i_0) = \tilde \cN_1$ and  $\tilde \cN_1$ is torsion free. Hence, by Lemma~\ref{reflexive and torsion free}, we obtain:
\begin{equation*}
\codim_X\Supp(\cext^1(\coker(i_k),\cO_X)) \geq 2
\end{equation*}
and so
\begin{equation*}
\codim_X\Supp(\coker(i_k^*)) \geq 2,
\end{equation*}
that is, the $i_k$ satisfy the condition \eqref{cond:square} \eqref{sq2}.

Second, let us consider the exact sequence
\begin{equation*}
0\To\coker(i_k)\To\coker(i_{k-1})\To
\coker(\tN_{k-1}\oplus\tN_{k+1} \xrightarrow{}\tN_k)\To 0.
\end{equation*}
{}From it we obtain an exact sequence
\begin{multline}
\cext^1(\coker(\tN_{k-1}\oplus\tN_{k+1} \xrightarrow{}\tN_k) ,\cO_X)\to \cext^1(\coker(i_{k-1}),\cO_X)
\\
\to\cext^1(\coker(i_k),\cO_X)\to\cext^2(\coker(\tN_{k-1}\oplus\tN_{k+1} \xrightarrow{}\tN_k) ,\cO_X).
\end{multline}
Now,
by \eqref{cond:square} \eqref{sq4}, we conclude that,
for a sufficiently large integer $k_0$,
$$\text{$\cext^\nu(\coker(\tN_{k-1}\oplus\tN_{k+1}
\xrightarrow{}\tN_k) ,\cO_X)=0$ for $\nu=0,1,2$ and $k\ge k_0$,}$$
and hence
\begin{equation*}
 \cext^1(\coker(i_{k-1}),\cO_X)
\isoto
\cext^1(\coker(i_k),\cO_X)\quad\text{for $k\ge k_0$.}
\end{equation*}
{}From \eqref{auxiliary for duals of i_k} we conclude that
a sequence $\{\coker(i_{k}^*)\}_{k\ge k_0}$ is an increasing sequence of
coherent subsheaves of $\cext^1(\coker(i_{k_0}),\cO_X)$.
Possibly by shrinking $X$, we see that
\begin{equation}
\begin{gathered}
\text{There exists an integer $k_1$ such that}
\\
\qquad\text{$\coker(i_{k_1}^*)\isoto\coker(i_{k}^*)$ for $k\geq k_1$}\,.
\end{gathered}\label{st:coki}
\end{equation}
The following diagram dual to \eqref{diag:square}
has also exact rows and exact columns:
\eqn
&&\hs{2ex}\xymatrix@C=2.5ex{
&0\ar[d]& 0\ar[d]\\
0\ar[r]&\ker(i_{k-1}^*)\ar[r]\ar[d]&\ker(i_k^*)\ar[r]\ar[d]&0\ar[r]\ar[d]&0\\
0\ar[r]&\tN_k^* \ar[d]^{i_{k-1}^*}\ar[r]^{p_k^*}&\tN_{k+1^*}\ar[r]
\ar[d]^{i_k^*}&\coker(p_k^*)
\ar[r]\ar[d]&0\\
0\ar[r]&\tN_{k-1}^*\ar[d]\ar[r]^{p_{k-1}^*}&\tN_k\ar[r]\ar[d]&\coker(p_{k-1}^*)
 \ar[r]\ar[d]&0\\
0\ar[r]&\coker(i_{k-1}^*)\ar[d]\ar[r]&\coker(i_k^*)\ar[d]\ar[r]&
\coker(\tN_{k-1}^*\oplus\tN_{k+1}^*\xrightarrow{}\tN_k^*)\ar[d]\ar[r]&0\\
&0&0&0
}
\eneqn
Hence \eqref{st:coki} implies that
\eq
&&\text{$\coker(\tN_{k-1}^*\oplus\tN_{k+1}^*\xrightarrow{}\tN_k^*)\simeq0$
locally for $k\gg0$.}\label{eq:N_k}
\eneq
In particular the dual diagram
\eqref{Dual Cartesian} satisfies \eqref {cond:square} \eqref{sq4}.

We can now reverse this process and start from the dual diagram

\begin{equation}
\label{dual diagram}
\begin{CD}
\tN_k^*@>{p_k^*}>> \tN_{k+1}^*
\\
@V{i_{k-1}^*}VV @VV{i_k^*}V
\\
\tN_{k-1}^* @>{p_{k-1}^*}>> \tN_k^*\,.
\end{CD}
\end{equation}

Dualizing it we obtain our original diagram

\begin{equation*}
\begin{CD}
\tN_k @>{i_k}>> \tN_{k+1}
\\
@V{p_{k-1}}VV @VV{p_k}V
\\
\tN_{k-1} @>{i_{k-1}}>> \tN_k
\end{CD}
\end{equation*}

All the hypotheses \eqref{cond:square}
 are satisfied for the dual diagram \eqref{dual diagram}, and therefore
the dual statement of \eqref{eq:N_k} holds, namely
\begin{equation*}
\coker(\tN_{k-1}\oplus\tN_{k+1} \xrightarrow{}\tN_k)=0 \quad
\text{locally for $k\gg0$.}
\end{equation*}
Therefore the right column in \eqref{diag:square} implies that
$\coker(i_k)\to \coker(i_{k-1})$ is an isomorphism for $k\gg0$.
Thus we established \eqref{stability of coker},
and the proof of Theorem~\ref{5} is complete.

\section{The commutative formal submodule extension theorem}
\label{The formal submodule theorem}

In this section we give a proof of Theorem~\ref{9}.
Let us recall set-up. We consider a complex manifold $X$ and
a subvariety $Y$ of $X$ of codimension at least two
and we write $j\cl X\setminus Y \to X$ for the inclusion.
Let $\hN$ be a coherent $\hcA_X$-module,
$\hL$ a torsion free coherent $\hcA_{X\setminus Y}$-module
and let $\vphi\cl j^{-1}\hN\epito\hL$ be an epimorphism  of
 $\hcA_{X\setminus Y}$-modules.
We will show that the image of $\hN\to j_*\hL$ is a coherent $\hcA_X$-module.

\medskip
Set
$\hL'\seteq\odh\odh\hL$.
Then $\hL'$ is a reflexive coherent $\hcA_{X\setminus Y}$-module
and we have a monomorphism
$\hL\monoto \hL'$  because  $\hL$ is torsion free $\hcA_{X\setminus Y}$-module.

We first observe:
\begin{equation}
\text{$\hL/(\hL \cap t^k\hL')$ is torsion free coherent $\cO_{X\setminus Y}$-module.}
\end{equation}
The coherency is easily deduced
for example from the fact that $\hL$,
$\hL'$ are all lattices. To argue that it  is torsion  free we note that
$\hL/(\hL \cap t^k\hL')\subset \hL'/t^k\hL'$
and thus we are reduced to showing that  $\hL'/t^k\hL'$ is a
torsion free $\cO_{X\setminus Y}$-module.
This is a general fact about reflexive $\hcA$-modules
which can be argued directly
but also follows immediately from \eqref{second basic exact sequence}.

Set $\hN_k=\hN/t^k\hN$. Then we have an epimorphism
$j^{-1}\hN_k\epito \hL/(\hL\cap t^k\hL')$.
We now apply  the Siu-Trautmann Theorem~\ref{Siu Trautmann} and conclude that
$\tN_k\seteq\operatorname{Im}\bl\hN_k \to j_*(\hL/(\hL \cap t^k\hL'))\br$ is coherent.
Then  the morphism $\hN\to j_*\cL$ decomposes
into the composition of $\hN\to \varprojlim \tN_k$ and a monomorphism
$$\varprojlim \tN_k
\monoto\varprojlim j_*(\hL/(\hL \cap t^k\hL'))
\simeq j_*\varprojlim \hL/(\hL \cap t^k\hL')
\simeq j_*\hL.$$
Hence we have reduced the problem to
the coherency of $\varprojlim \tN_k$.
In order to see this,
we proceed as in the previous section and appeal to
\cite[Proposition 1.2.18]{KS3}.

We first decompose $\tN_k \xrightarrow{t}\tN_k$ into a composition of two maps just as in the previous section. We begin with
the commutative diagram
\eq&&
\ba{l}
\xymatrix{
\hN_k \ar[r]^{i_k}\ar[d]_{p_{k-1}}\ar[dr]|-{t\rule[-.5ex]{0ex}{2ex}}&\hN_{k+1} \ar[d]^{p_{k}}\\
\hN_{k-1}\ar[r]^{i_{k-1}}&\hN_k
}\ea
\label{first diagram}
\eneq
where we have written $i_k\cl \hN_k \to \hN_{k+1}$ for the multiplication map $t\cl\hN_k \to \hN_{k+1}$ and $p_k \cl \hN_{k+1} \to \hN_{k}$ for the natural projection $\hN_{k+1}\to \hN_k$.
The maps $p_k$ are, obviously, epimorphisms. We have a similar commutative diagram on $X\setminus Y$
\begin{equation*}
\xymatrix{
\hL/(\hL \cap t^k\hL')\ar[r]^{i_k}\ar[d]_{p_{k-1}}\ar[dr]|-{t\rule[-.5ex]{0ex}{2ex}}
&\hL/(\hL \cap t^{k+1}\hL')
\ar[d]^{p_k}\\
\hL/(\hL \cap t^{k-1}\hL')\ar[r]^{i_{k-1}}&\hL/(\hL \cap t^k\hL')
}
\end{equation*}
where the $i_k$ are monomorphisms and the $p_k$ are epimorphisms. Applying $j_*$ to this diagram we obtain
\begin{equation}
\label{second}
\begin{CD}
j_*(\hL/\hL \cap t^k\hL') @>{i_k}>>j_*(\hL/\hL \cap t^{k+1}\hL')
\\
@V{p_{k-1}}VV @VV{p_k}V
\\
j_*(\hL/\hL \cap t^{k-1}\hL')@>{i_{k-1}}>>j_*(\hL/\hL \cap t^k\hL')
\end{CD}
\end{equation}
where, by the left exactness of $j_*$, the $i_k$ are still monomorphisms. Taking the image of the diagram \eqref{first diagram}
to the diagram \eqref{second} we obtain a commutative diagram
\begin{equation}\ba{l}
\xymatrix{
\tN_k\ \ar@{>->}[r]^{i_k}\ar@{->>}[d]_{p_{k-1}}\ar[dr]|-{t}
&\tN_{k+1} \ar@{->>}[d]^{p_{k}}\\
\tN_{k-1}\ \ar@{>->}[r]^{i_{k-1}}&\tN_k
}\ea
\end{equation}
where the $i_k$ are monomorphisms and the $p_k$ are epimorphisms.

In order to see that $\prolim\tN_k$ is a coherent $\hcA_X$-module,
it suffices to prove the following two statements:
\begin{subequations}
\begin{equation}
\label{submodule ker stability}
\text{$\ker(\tN_{k} \To[i_k]\tN_{k+1})\to \ker(\tN_{k-1} \To[i_{k-1}]\tN_{k})$
is an isomorphism for $k\gg0$,}
\end{equation}

\begin{multline}
\label{submodule coker stability}
\parbox{65ex}{$\coker(\tN_{k} \To[i_k]\tN_{k+1})\to \coker(\tN_{k-1} \To[i_{k-1}]\tN_{k})$
is an isomorphism for $k\gg0$,}
\end{multline}
\end{subequations}
by \cite[Proposition 1.2.18]{KS3}.
The first stability \eqref{submodule ker stability} is obvious.
Let us show \eqref{submodule coker stability}.
We write $\cF_k= \ker(p_k)$ and then we have:
\begin{equation*}
\begin{CD}
@. 0 @. 0 @. 0 @.
\\
@. @VVV @VVV @VVV@.
\\
0 @>>> \cF_{k-1} @>>> \cF_k@>>> \cF_k/\cF_{k-1} @>>> 0
\\
@. @VVV @VVV @VVV @.
\\
0 @>>> \tN_k @>{i_k}>> \tN_{k+1} @>>> \coker(i_k)@>>> 0
\\
@. @V{p_{k-1}}VV @VV{p_k}V @VV{p'_k}V @.
\\
0 @>>> \tN_{k-1} @>{i_{k-1}}>> \tN_k @>>> \coker(i_{k-1}) @>>> 0
\\
@. @VVV @VVV @VVV @.
\\
@. 0 @>>> 0 @>>> \coker(p'_k) @>>> 0
\\
@. @. @. @VVV @.
\\
 @.@.@.0 @.
\end{CD}
\end{equation*}
{}From this commutative diagram with exact rows and columns we conclude
that it is enough to show
that $\cF_{k-1}\to\cF_{k}$ is an isomorphism for $k\gg0$.
We also note that the  $\cF_k$ are torsion free $\cO_X$-modules
as they are submodules of the torsion free $\cO_X$-module $\tN_k$.

Let us first work locally outside of $Y$. Outside of $Y$ we have
\begin{equation}
\cF_k|_{X\setminus Y} \ = \ \ker\bl\dfrac{\hL}{\hL \cap t^{k+1}\hL'}
 \to\dfrac{\hL}{\hL \cap t^{k}\hL'}
\br=\frac {\hL \cap t^{k}\hL'}{\hL \cap t^{k+1}\hL'}.
\end{equation}
We now identify
\begin{equation}
\cF_k|_{X\setminus Y} \ = \ \frac {t^{-k}\hL \cap \hL'}{t^{-k}\hL \cap t \hL'}
\subset \frac { \hL'}{ t \hL'} \,.
\end{equation}
Under this identification the map $t\cl \cF_{k-1} \to \cF_{k}$ becomes an inclusion and we get an increasing family of coherent subsheaves of
${ \hL'}/{ t \hL'}$. By the Noetherian property this sequence
stabilizes locally on $X\setminus Y$.

Let us now work in the neighborhood of a point $y\in Y$.
 We then  proceed in the same manner as in the last section.
We write $X$ locally as a product $X= D_m \times D_\ell$, where
$D_m$ is an $m$-dimensional  ball of radius 2 in $\bC^m$,
$Y=\{0\}\times D_\ell$ and $y=(0,0)$.
Let us write $\pi\cl X\to D_\ell$ for the projection
with respect to this decomposition.
We consider a relatively compact open neighborhood $K$ of $y$ of the form:
\begin{equation}
\begin{gathered}
K\ = \set{z\in D_m }{\|z\|<1}  \times U\,,
\end{gathered}
\end{equation}
where $U$ is a relatively compact open neighborhood of $y\in D_\ell$.
Since the $\cF_k$ stabilize on any compact subset of $X\setminus Y$,
the projection  $\Supp(\cF_{k}/\cF_{k-1})\cap(W\times U)
\to U$ is a finite map and this implies that
\begin{equation*}
\text{$\codim_X\Supp(\cF_{k}/\cF_{k-1})\cap(W\times U) \geq 2$ for $k\geq k_0$.}
\end{equation*}
Hence shrinking $X$ if necessary, we may assume that
$\codim_X\Supp(\cF_{k}/\cF_{k-1}) \geq 2$ for $k\ge k_0$.
It implies that $\cF_{k-1}^{**}\to \cF_{k}^{**}$ is an isomorphism
for
$k\ge k_0$.
As $\cF_k$ is torsion free, $\cF_k\subset \cF_k^{**}$.
Thus, again by the Noetherian property, the
increasing sequence $\cF_k \subset \cF_{k_0}^{**} $ locally stabilizes. This concludes the argument.

 \section{Comparison of the formal and convergent cases}
 \label{comparison}

  In this section we prove Theorem \ref{8}. Let us recall the statement. We are given a locally free $\cA_{X-Y}$-module $\cM$
of finite rank on $X-Y$ and we assume that
$\hM\seteq\hcA_{X-Y}\otimes_{\cA_{X-Y}}\cM$ extends
to a locally free $\hcA_{X}$-module
 to all of $X$. We are to show that  if $\dim Y \leq \dim X -2$, then $\cM$ also extends to locally free $\cA_{X}$-module on all of $X$.

 Let us begin with a reduction.
 The question being local, we can assume, proceeding by induction,
that we are working in the neighborhood of a smooth point $y$ of $Y$.
Furthermore, we can and will assume that the codimension of $Y$
is precisely two. We can do so because if $Y$ is not of codimension two
we can always replace it in the neighborhood of $y$
by a larger codimension two submanifold.
Furthermore, we may assume that
 $\hM$ is a free $\hcA_{X\setminus Y}$-module.

Thus, we are reduced to  the following situation. Let us write $\bC^m =  \bC^2 \times \bC^{m-2}$. Let us consider a small neighborhood
$$W = \set{(z_1,\dots, z_m)\in \bC^m}{|z_i| < \rho \  \text{for all}\  1\leq i\leq m}$$ of the origin.
We consider $W$ as a neighborhood of a smooth point $y\in Y$ in $X$ such that $y$ corresponds to the origin and $Y$ in the neighborhood of $y$ to the locus $(\{0\} \times  \bC^{m-2})\cap W$. Let us write $Z =  (\{0\} \times  \bC^{m-2})\cap W$. We now restate out hypotheses in this context. We are given a locally free $\cA_{W-Z}$-module $\cM$ on $W - Z$ such that the corresponding $\hcA_{W-Z}$ module   $\hM$ is trivial on $W - Z$.  If we can show that the locally free module $\cM$ is trivial on $W - Z$ then it of course would extend. This is probably true, but we will prove a slightly weaker statement where we shrink the neighborhood $W$.  This is harmless for our purposes. Let us then write
\begin{equation}
N \ = \ \{(z_1,\dots, z_m)\in \bC^m \mid |z_i| < \rho' \}\subset W
\end{equation}
for $\rho'<\rho$ and we consider the following family of open subsets of $N$:
\begin{equation}
N_\delta \ = \ N \setminus
\set{(z_1,\dots, z_m)\in \bC^m}{\mid |z_1| \leq \delta, \ |z_2| \leq \delta}\,;
\end{equation}
where we vary $\delta$ with $0\le\delta<\rho'$.
Note that $\bar N\subset W$ and $\bar N_\delta  \subset W\setminus Z$ are compact
for $\delta>0$.

{}From the discussion above we conclude that Theorem~\ref{8} follows from the claim:\begin{equation}
\label{first reduction}
\text{$\cM$ is trivial on $N_0$.}
\end{equation}
We first observe that:
\begin{equation}
\label{reduction to relatively compact}
\text{$\cM$ is trivial on $N_0$ if it is trivial on $N_\delta$ for all sufficiently small $\delta>0$}\,.
\end{equation}
Let us argue this. We write $d$ for the rank of $\cM$. For all small $\delta$ we have an isomorphism
\begin{equation*}
\cA_{N_\delta}^{\oplus d} \isoto[{\ f_\delta}] \cM|_{{N_\delta} }\,.
\end{equation*}
 For $\delta'<\delta$ we then have a map
 \begin{equation*}
 \cA_{N_\delta}^{\oplus d}  \xrightarrow
{f_{\delta'}^{-1}\circ f_\delta} \cA_{N_\delta'}^{\oplus d}\,,
 \end{equation*}
 which is defined and is an isomorphism on $N_\delta$. By Hartogs' theorem this map extends to all of $N$. Let us write $r_{\delta',\delta}$ for the resulting automorphism of $\cA_N^{\oplus d}$. Then we have
 \begin{equation*}
 f_\delta \ = \ f_{\delta'}\circ r_{\delta',\delta} \qquad \text{on} \ \ N_\delta\,.
 \end{equation*}
 As the right hand side of the formula is defined on $N_{\delta'}$ for any $0<\delta'<\delta$ we see that the map $f_\delta$ extends to an isomorphism
 \begin{equation*}
 \cA_{N_0}^{\oplus d}  \isoto[{ f_0}] \cM|_{{N_0} }\,.
 \end{equation*}
 Thus we have established \eqref{reduction to relatively compact}.

\medskip
 We will now work on a particular $N_\delta$ keeping the assumption that $\delta$ is small. This hypothesis will be used at some point for purely technical reasons.

 We now consider $\cM$ on $\bar N_\delta$. As $\hM$ is locally free,
$\cM$ is also a locally free $\cA_{W\setminus Z}$-module.
Hence there will be a finite open cover $V_i$ of
$\bar N_\delta$ so that  $\cM$ is trivial on $V_i$.
We choose particular trivializations on each of the $V_i$.
As the cover is finite, we can find a $C>0$ such that all the transition functions take values in $GL_d(A_C)$. We conclude
\begin{equation}
\parbox{58ex}{there exist $C>0$ and a locally free
$(\cA_X^C\vert_{N_\delta})$-module $\cM_C$ such that
$\cM|_{N_\delta}\simeq \cA_X\vert_{N_\delta}\otimes_{\cA_X^C\vert_{N_\delta}}\cM_C$.}
\end{equation}

 We will be working with transition functions and to this end we need to set up some notation.  Let us consider the groups $GL_d(A_C)$, $GL_d(A)$, and $GL_d(\hA)$. As usual, we will write $\gl_d$ for $d\times d$-matrices and consider it as a group via its usual additive structure.

Let us introduce the Banach algebra  $\gl_d(A_C)$.
Pick $a\in \gl_d(A_C)$. We write
 \begin{equation}
 a \ = \ \sum_{n=0}^\infty a_n t^n
\qquad\text{with $ a_n\in \gl_d(\bC)$ ($n\ge0$).}\label{eq:a exp}
 \end{equation}
We now define a  norm on $\gl_d(A_C)$ in the following manner
\begin{equation}
\| a \|_C \ = \ \sum_{n=0}^\infty \|a_n \|\frac {C^n}{n!}\qquad \text{where  $\|a_n \|$ is the operator norm on $\gl_d(\bC)$} \,.
\end{equation}
Thus, we have
\begin{equation}
\text{for $a\in\gl_d(A)$,
$a\in \gl_d(A_C)$ if and only if $\| a \|_C< \infty$.}
\end{equation}
Therefore $\gl_d(A_C)$ is a Banach algebra and $GL_d(A_C)$
consists of units in the Banach algebra $\gl_d(A_C)$.
Also,  $\gl_d(A)$ is a topological algebra and
we can view $GL_d(A)$ as units in $\gl_d(A)$.
We have, of course,
 \begin{equation}
\gl_d(A) \ = \  \varinjlim_{C>0}\gl_d(A_C), \qquad GL_d(A) \ = \  \varinjlim_{C>0} GL_d(A_C).
 \end{equation}

It follows from Lemma~\ref{lem:Banachring} and Proposition~\ref{structure of A} that
\begin{equation}
\text{The rings $A_C$ and $A$ are local rings}\,.
\end{equation}
Thus, it is easy to describe the units.
Let $a\in \gl_d(A_C)$. We write it as in $\eqref{eq:a exp}$.
 Then
 \begin{equation}
\text{$a\in GL_d(A_C)$ if and only if $a_0\in GL_d(\bC)$}\,.
\end{equation}
Let us consider the canonical map $GL_d(A_C)\to GL_d(\bC)$ mapping $a\mapsto a_0$. Let us write
\begin{equation*}
\Gamma_C \ = \ \ker (GL_d(A_C)\to GL_d(\bC))\,.
\end{equation*}
We then obtain a  semidirect  product
\begin{equation*}
GL_d(A_C) \ = \   \Gamma_C  \rtimes  GL_d(\bC)\,.
\end{equation*}
Now,
\begin{equation*}
\Gamma_C \ = \ \{1 + \sum_{n=1}^\infty a_n t^n \mid \sum_{n=1}^\infty \|a_n\|\frac {C^n}{n!} <\infty\}
\end{equation*}
and hence $\Gamma_C$ is contractible and therefore $GL_d(A_C)$ and $GL_d(\bC)$ have the same homotopy type. Thus we conclude
\begin{prop}
On a complex manifold there is natural bijection between topological rank $d$ vector bundles and topological $A_C$-bundles of rank $d$.
\end{prop}

We will now make use of a result of Bungart.
Let B be a Banach algebra and let us write $B^\times$ for the units in $B$. Then
\begin{thm}
\label{A bundles on Stein}
On a Stein space there  is a natural bijection between isomorphism classes of holomorphic $B^\times$-bundles and topological $B^\times$-bundles.
\end{thm}
For a proof, see  \cite[Theorem 8.1]{Bu}.

We will apply this result for $B= \gl_d(A_C)$ in which case $B^\times=GL_d(A_C)$.
Set
\begin{equation}
U_i \ = \set{(z_1,\dots, z_N)\in N_\delta}{|z_i|>\delta} \qquad \text{for $i = 1, 2$.}
\end{equation}
The $U_i $ ($i=1,2$) form a Stein cover of $N_\delta$.

As the $U_i$ have the homotopy type of a circle all topological complex vector bundles on them are trivial. Making use of the proposition and theorem above we conclude:
\begin{equation}
\text{The restrictions $\cM_C|_{U_i}$ are trivial}\,.
\end{equation}

We now trivialize the bundle $\cM_C$ on $U_1$ and $U_2$. This way we obtain an element $c\in GL_d(\cA^C_{N_\delta}(U_1\cap U_2))$, i.e., $c$ is a holomorphic function on $U_1\cap U_2$ with values
in $GL_d(A_C)$. As the corresponding formal bundle $\hM$ is trivial on $N_\delta$ this class is trivial in $\oh^1(N_\delta, GL_d(\hcA_{N_\delta}))$ and therefore there are elements $a\in GL_d(\hat\cA_{N}(U_1))$ and $ b\in GL_d(\hat\cA_{N}(U_2))$ such that $c = ab$. The elements $a$ and $b$ are unique up to an element $e\in GL_d(\hat\cA_{N}({N_\delta}))$, i.e., we can replace the pair $(a,b)$ by $(ae, e^{-1}b)$. Note that, by Hartogs' theorem $GL_d(\hat\cA_{N}({N_\delta})) = GL_d(\hat\cA_{N}({N}))$ and so we view $e\in GL_d(\hat\cA_{N}({N}))$.

Our goal is to choose the $e\in GL_d(\hat\cA_{N}({N}))$ in such a way that $ae\in GL_d(\cA_{N}(U_1))$ and $e^{-1}b\in GL_d(\cA_{N}(U_2))$ showing that the bundle $\cM$ is trivial on ${N_\delta}$. In this process we will replacing the $C$ by a smaller constant.

  We write
 \begin{equation*}
 \begin{aligned}
&c=\sum_{n=0}^\infty c_nt^n \qquad c_0\in GL_d(\cO_N(U_1\cap U_2)) \ \  c_n\in \gl_d(\cO_N(U_1\cap U_2)) \ \text{for $n\geq 1$,}
 \\
&a=\sum_{n=0}^\infty a_nt^n \qquad a_0\in GL_d(\cO_N(U_1)) \ \  a_n\in \gl_d(\cO_N(U_1)) \ \text{for $n\geq 1$,}
\\
&b=\sum_{n=0}^\infty b_nt^n \qquad b_0\in GL_d(\cO_N(U_2)) \ \  b_n\in \gl_d(\cO_N(U_2)) \ \text{for $n\geq 1$.}
 \end{aligned}
 \end{equation*}
 The fact that $c=ab$ amounts to the equations
\begin{equation}
c_n = \sum _{i=0}^n a_i b_{n-i}\,.
\end{equation}

As a first step let us deal with the first terms $a_0$, $b_0$, and $c_0$. We replace the term  $c$ by $a_0^{-1}cb_0^{-1}$, the term $a$ by $a_0^{-1}a$, and the term $b$ by $bb_0^{-1}$.  This modification reduces us to the situation where
\begin{equation}
a_0 = 1, \ \ b_0 = 1, \ \ c_0 =1 \,.
\end{equation}
We will now expand the term $a$ in the $z_1$ coordinate and the term  $b$ in the $z_2$ coordinate.
\begin{equation}
\begin{gathered}
a= a^{+}+a^{-}, \qquad  a=\sum _{\ell=-\infty}^\infty g_\ell(z_2, \dots,z_m) z_1^\ell,
\\
a^+ = \sum _{\ell\ge0} g_\ell(z_2, \dots,z_m) z_1^\ell, \quad a^{-} =\sum _{\ell<0} g_\ell(z_2, \dots,z_m) z_1^\ell
\end{gathered}\label{def:a+}
\end{equation}
and
\begin{equation}
\begin{gathered}
b= b^{+}+b^{-}, \qquad  b=\sum _{\ell=-\infty}^\infty h_\ell(z_1,  \hat z_2,\dots,z_m) z_2^\ell,
\\
b^+ = \sum _{\ell\ge0} h_\ell(z_1,  \hat z_2,\dots,z_m) z_2^\ell ,
\quad b^{-} =\sum _{\ell<0}  h_\ell(z_1,  \hat z_2,\dots,z_m) z_2^\ell.
\end{gathered}
\end{equation}

We now choose the element  $e$. We will make use of it in the following form:
\begin{equation*}
e = \prod _{n=1}^\infty (1+e_nt^n)=(1+e_1t)(1+e_2t^2)(1+e_3t^3)\cdots\;,
\qquad e_n\in  \gl_d(\cO_N(N))\,.
\end{equation*}
With a suitable choice of $\{e_k\}_{k\ge1}$, we define
\begin{equation*}
a(k) \in GL_d(\hat\cA_N(U_1),)  \quad b(k) \in GL_d(\hat\cA_N(U_2))
\end{equation*}
such that
\begin{equation*}
a(k+1) = a(k) (1+ e_k t^k) \qquad b(k+1) = (1+ e_k t^k)^{-1} b(k)\,,
\end{equation*}
with the initial conditions $a(1)=a$, $b(1)=b$.
Then we can  easily see  that
$$a(m)_k=a(k+1)_k=a(k)_k+e_k\quad\text{for $m\ge k+1$.}$$

We make the following choice
\begin{equation*}
e_k  =  -a(k)_k^+ \,.
\end{equation*}
The end result is as follows:
\begin{equation}
ae \ \ \text{has the property that for all $n\geq 1$} \ \ \  (ae)_n^{+}\ = \ 0\,.
\end{equation}
Here $a^+=\sum_{n=0}^\infty a_n^+t^n$ is given in \eqref{def:a+}.

Thus, after replacing $a$ by $ae$ and $b$ by $e^{-1}b$ we have
\begin{equation}
\label{crucial}
c \ = \ ab, \qquad a_n^{+}\ = \ 0  \ \ \  \text{for all $n\geq 1$} \,.
\end{equation}
We claim that after this modification both $a$ and $b$ are convergent:
\begin{prop}
If $c\in GL_d(\cA_N(U_1\cap U_2))$ can be written as $c=ab$ with $a\in GL_d(\hcA_N(U_1))$ and  $b\in GL_d(\hcA_N(U_2))$ such that $a_n^+ = 0$ for $n\geq 1$ then  $a\in GL_d(\cA_N(U_1))$ and  $b\in GL_d(\cA_N(U_2))$\,.
\end{prop}
In particular this proposition implies that the bundle associated to the cocycle $c$ is trivial thus completing the proof. The rest of this section is devoted to the proof of this proposition.

In order to prove the lemma we need to make estimates on an exhaustive family of compact sets. First let us write
\begin{equation}
K \ = \  \{(z_1,\dots, z_m)\in \bC^m \mid |z_i| \leq R \}\subset N
\end{equation}

We choose the families in the following manner
\begin{equation*}
\begin{gathered}
K_1 \ = \ \{(z_1,\dots, z_m)\in K \mid 0 < r \leq |z_1|\} \subset U_1,
\\
K_2 \ = \ \{(z_1,\dots, z_m)\in K \mid 0 < r \leq |z_2|\} \subset U_2,
\\
K_{12}=K_1\cap K_2 = \{(z_1,\dots, z_m)\in K \mid  0 < r \leq |z_1|  \ \ 0 < r \leq |z_2|\} \subset U_1\cap U_2\,.
\end{gathered}
\end{equation*}

In the arguments that follow we make the convention that
\begin{assumption}
\label{large}
We will from now on assume that  $R/r\geq 16$. We can always achieve this by enlarging the compact sets $K_i$ by shrinking the $r$.
\end{assumption}
Note that shrinking the $r$ also force as to consider only those $\delta$ that are sufficiently small.

\begin{rmk}
The choice $R/r\geq 16$ is of course rather arbitrary and simply depends on the way we write the estimates in Lemma ~\ref{cousin estimate}.
\end{rmk}

Let us recall our norm $\|\ \|$ for $\gl_d(A_C)$. We have set
\begin{equation*}
\|f\| \ = \ \sum_{n=0}^\infty \|f_n\| \frac {C^n}{n!}
\quad\text{for $f=\sum_{n=0}^\infty f_n t^n \in \gl_d(A_C)$,}
\end{equation*}
where $\|f_n\|$ denotes the operator norm of $f_n\in \gl_d(\bC)$.

Given a holomorphic function $h\cl U \to  \gl_d(\bC)$ on an open set $U$,
 its norm $\|h\|$ stands for the continuous function on $U$ whose value at $x\in U$ is $\|h(x)\|$.  If  $K\subset U$ is compact,
we write $\|h\|_K$ for the sup norm of $\|h\|$ on $K$, i.e.,
\begin{equation}
\|h\|_K \ = \ \sup_{x\in K} \|h(x)\|\,.
\end{equation}
We will be using this notation for the elements $a$, $b$, and $c$.
As analytic functions achieve their maximum on the boundary of the region,
we make the following observations:
\begin{equation}
\begin{gathered}
\text{maximum on $K_1$ is achieved on}
\\
\{ |z_1| =r,  |z_i| = R\  \text{for}\   i\neq 1 \}\cup\{ |z_i|= R \  \text{for all}\ i \}\,.
\end{gathered}
\end{equation}
Similarly we see that
\begin{equation}
\begin{gathered}
\text{maximum on $K_2$ is achieved on}
\\
\{ |z_2| =r,  |z_i| = R\  \text{for}\   i\neq 2 \}\cup\{ |z_i|= R \  \text{for all}\ i \}\,.
\end{gathered}
\end{equation}
Finally, the maximum on $K_{12}=K_1\cap K_2$ is achieved on
\begin{equation*}
\begin{gathered}
\ \ \{ |z_1| =r, |z_2|=r, \ |z_i| = R\  \text{for}\   i\neq 1,2\}\cup \{ |z_1| =r,  |z_i| = R\  \text{for}\   i\neq 1 \} \cup
\\
\{ |z_2| =r,  |z_i| = R\  \text{for}\   i\neq 2 \}\cup\{ |z_i|= R \  \text{for all}\ i \}\,.
\end{gathered}
\end{equation*}

\begin{rmk}
We will make crucial use of the fact that the maximum on $K_1$ and $K_2$ are achieved on $K_{12}$.
Hence we have $$\|h\|_{K_1} =\|h\|_{K_{12}}\quad\text{for a holomorphic function $h$
defined on $U_1$.}$$
$$\|h\|_{K_2} =\|h\|_{K_{12}}\quad\text{for a holomorphic function $h$
defined on $U_2$.}$$

This allows us to compare the norms on different compact sets.
\end{rmk}

Let us now consider the element  $c\in GL_d(\cA_N(U_1\cap U_2))$. Given the compact set $K_{12}$ there exists a $D>0$ such that
\begin{equation}
\|c_n\|_{K_{12}} \ \leq D^n n! \ \ \ \text{for all} \ n\,.
\end{equation}
Note that our element $c$ actually lies in $GL_d(\cA^C_N(U_1\cap U_2))$ but we do not actually need to use this fact here as we argue on a family of compact sets.
We need to show that  we can find a possibly larger $\tilde D$ such that
\begin{equation}
\label{required estimate}
\|a_n\|_{K_1} \ \leq \tilde D^n n!  \ \ \text{and} \ \  \|b_n\|_{K_2} \ \leq \tilde D^n n!\ \ \ \text{for all} \ n\,.
\end{equation}

We will be proceeding by induction which we will begin at some particular $n_0$ which will be chosen below. Let us now consider the induction step and so we assume that assume that we have the estimate \eqref{required estimate} up to $n-1$.

To this end, we consider the equation
\begin{equation}
c_n =  b_n + \sum _{i=1}^{n-1} a_i b_{n-i}+a_n\,.
\end{equation}
and use it bound the norms of $a_n$ and $b_n$.  Let us first observe that  by the remark above we have
\begin{equation}
\|a_i\|_{K_1} =\|a_i\|_{K_{12}},\quad \|b_i\|_{K_2} =\|b_i\|_{K_{12}}\,.
\end{equation}

Thus all the estimates in the rest of this section can be done on $K_{12}$.
We now have:
\begin{equation}
\| \sum _{i=1}^{n-1} a_i b_{n-i}\|_{K_{12}} \leq \sum _{i=1}^{n-1} \|a_i\|_{K_{12}} \|b_{n-i}\|_{K_{12}}  \leq \sum _{i=1}^{n-1} i! (n-i)! \tilde D^n
\end{equation}

 Let us write
 \begin{equation}
\epsilon_n\ = \  \frac {\sum _{i=1}^{n-1} i! (n-i)!} {n!} \ =
 \  \sum _{i=1}^{n-1}\frac{1} {\binom{n}{i}} \,.
\end{equation}
Then

\begin{equation}
\| \sum _{i=1}^{n-1} a_i b_{n-i}\|_{K_{12}}\leq \epsilon_n \tilde D^n n!
\end{equation}
and we conclude that
\begin{equation}
\label {induction step}
\| c_n - \sum _{i=1}^{n-1} a_i b_{n-i}\|_{K_{12}}\leq D^n n! + \epsilon_n \tilde D^n n! = (D^n  + \epsilon_n \tilde D^n) n!.
\end{equation}

Note that:
 \begin{equation}
\lim _{n \to \infty}\epsilon_n\ = \ 0 \,.
\end{equation}

Thus, given any $\epsilon$ with $0<\epsilon <1$ we will choose $n_0$ and $\tilde D$ such that
\begin{subequations}
\label{choice}
\begin{equation}
\text{\eqref{required estimate} is satisfied for all $n \leq n_0$,}
\end{equation}
\begin{equation}
\text{for $n\geq n_0$ we have  $\epsilon_n<\frac\epsilon 2$, }
\end{equation}
and
\begin{equation}
\begin{gathered}
\text{$\tilde D > \frac {2 D} \epsilon$.}
\end{gathered}
\end{equation}
\end{subequations}
With these choices we can now prove the estimate \eqref{required estimate} by induction beginning with $n_0$ and making use of \eqref{induction step}. So we assume that the estimate \eqref{required estimate} has been proved up to $n-1$. Then
\begin{equation}
\| a_n +b_n\|_{K_{12}} = \| c_n - \sum _{i=1}^{n-1} a_i b_{n-i}\|_{K_{12}}\leq(D^n  + \epsilon_n \tilde D^n) n!
 \,.
\end{equation}
But now, using \eqref{choice}, we see that
\begin{equation}
D^n  + \epsilon_n \tilde D^n < \epsilon \tilde D^n\,.
\end{equation}
Thus, we have
\begin{equation}
\| a_n +b_n\|_{K_{12}}\leq   \epsilon\tilde D^{n} n!\;.
\end{equation}
Now, using \eqref{crucial}, i.e., the fact that $a_n^+=0$ for $n\geq 1$ we conclude, for $n \geq 1$, that
\begin{equation}
\label{induction bound}
\| a_n^- + b_n^-  + b_n^+\|_{K_{12}}\leq  \epsilon\tilde D^{n} n!\;.
\end{equation}
We now apply the following
\begin{lem}
\label{cousin estimate}
 Set   $f_n = b_n^+ +b_n^- + a_n^-$. Then there exists a universal constant $E\leq 8$, only depending on ${K_{12}}$, such that
\begin{equation*}
\|b_n^+ \|_{K_{12}} \leq E \|f_n\|_{K_{12}},
\quad \|b_n^- \|_{K_{12}} \leq E\|f_n\|_{K_{12}},
\quad \|a_n^-\|_{K_{12}}\leq E\|f_n\|_{K_{12}}\,.
\end{equation*}
\end{lem}
Let us first argue that this lemma will finish our induction argument. It immediately  implies that
\begin{equation}
\|b_n^+ \|_{K_{12}} \leq E\epsilon\tilde D^{n} n!\;,
\quad \|b_n^- \|_{K_{12}} \leq E\epsilon\tilde D^{n} n! \;,
\quad \|a_n^-\|_{K_{12}} \leq E\epsilon\tilde D^{n} n!\;;
\end{equation}
and furthermore that
\eqn&&
\|b_n \|_{K_{2}}=\|b_n \|_{K_{12}} \leq \|b_n^+ \|_{K_{12}} +\|b_n^- \|_{K_{12}} \ \leq 2E\epsilon\tilde D^{n} n!\;,
\eneqn
and
\eqn  \|a_n^-\|_{K_{1}}= \|a_n^-\|_{K_{12}}
\leq E\epsilon\tilde D^{n} n!\;.
\eneqn
Thus, any choice of $\epsilon< \frac{1}{2E}$ allows us to obtain \eqref{required estimate} and hence  our induction is complete.

It remains to prove the lemma.

\begin{proof}[Proof of Lemma~\ref{cousin estimate}]
We first observe that it is enough to consider the case $m=2$ as it is enough to prove estimates for fixed $z_3,\dots, z_m$ with $|z_i|=R$.
So, we let $m=2$.

We make the following observations:
\begin{subequations}
\begin{equation}
\begin{gathered}
\text{The term $\|b_n^+(z)\|$ obtains its maximum $\|b_n^+\|_{K_{12}}$  on the locus}
\\
 \text{where $|z_1|=R$ and  $|z_2|=R$,}
\end{gathered}
\end{equation}
\begin{equation}
\begin{gathered}
\text{The term $\|b_n^-(z)\|$ obtains its maximum $\|b_n^-\|_{K_{12}}$  on the locus}
\\
\text{where $|z_1|=R$ and $|z_2|=r$,}
\end{gathered}
\end{equation}
\begin{equation}
\begin{gathered}
\text{The term $\|a_n^-(z)\|$ obtains its maximum $\|a_n^-\|_{K_{12}}$  on the locus}
\\
\text{where $|z_1|=r$  and $|z_2|=R$.}
\end{gathered}
\end{equation}
\end{subequations}

We first assume that
\begin{equation}
\label{first case}
\begin{gathered}
 \|b_n^+\|_{K_{12}} \geq
\dfrac 1 4
\max\{\|b_n^-\|_{K_{12}},\|a_n^-\|_{K_{12}}\}.
\end{gathered}
\end{equation}

Then by using the Schwarz lemma we see that
\begin{equation}
\begin{gathered}
\text{On the set where $|z_1|=R$ and  $|z_2|=R$ we have}
\\
\|b_n^-(z)\| \leq \frac {r}{R} \|b_n^-\|_{K_{12}} \leq\frac {4\,r}{R}  \|b_n^+\|_{K_{12}}  \leq \frac 1 {4} \|b_n^+\|_{K_{12}}
\\
\|a_n^-(z)\| \leq \frac {r}{R} \|a_n^-\|_{K_{12}} \leq\frac {4\,r}{R}  \|b_n^+\|_{K_{12}} \leq \frac 1 {4} \|b_n^+\|_{K_{12}}
\,;
\end{gathered}
\end{equation}
the last inequalities follow from our Assumption~\ref{large}.
As $\|b_n^+(z)\|$ assumes is maximum on the set $|z_1|=|z_2|=R$,
we have
$$\|b_n^+\|_{K_{12}}=\|f_n-a_n^--b_n^-\|_{|z_1|=|z_2|=R}
\le \|f_n\|_{K_{12}}+(1/4+1/4)\|b_n^+\|_{K_{12}}.$$
We then conclude
\begin{equation}
 \|b_n^+\|_{K_{12}}\leq  2  \|f_n\|_{K_{12}}\,,
 \end{equation}
 and so then also by \eqref{first case}
  \begin{equation}
 \|b_n^-\|_{K_{12}}\leq  8  \|f_n\|_{K_{12}}\,, \quad \|a_n^-\|_{K_{12}}\leq  8 \|f_n\|_{K_{12}} \,.
\end{equation}
Let us now assume
\begin{equation}
\begin{gathered}
\|b_n^+\|_{K_{12}} <  \frac 1 4\max\{\|b_n^-\|_{K_{12}},\|a_n^-\|_{K_{12}}\}\,.
\end{gathered}
\end{equation}

We first assume that $\max\{\|b_n^-\|_{K_{12}},\|a_n^-\|_{K_{12}}\} = \|b_n^-\|_{K_{12}}$.
Then, again by utilizing the Schwarz lemma, we obtain
\begin{equation}
\label{b estimate for a}
\begin{gathered}
\text{On the set where $r\le|z_2|\le R$ and $|z_1|=R$  we have}
\\
\text{$\|a_n^-(z)\| \leq \frac {r}{R} \|a_n^-\|_{K_{12}} \leq  \frac {r}{R}  \|b_n^-\|_{K_{12}} \leq  \frac {1}{16}  \|b_n^-\|_{K_{12}}$,}
\end{gathered}
\end{equation}
where we have again made use of  assumption ~\ref{large}.
As $\|b_n^-(z)\|$ achieves its maximum on the set where $|z_2|=r$ and $|z_1|=R$,
 we have
$$\|b_n^-\|_{K_{12}}=\|f_n-a_n^--b_n^+\|_{|z_2|=r, |z_1|=R}
\le \|f_n\|_{K_{12}}+(1/16+1/4)\|b_n^-\|_{K_{12}}, $$
which implies
\begin{equation}
\|b_n^-\|_{K_{12}}\leq  2 \|f_n\|_{K_{12}}\,,
\end{equation}
 and hence  also
 \begin{equation}
\|b_n^+\|_{K_{12}}\leq  2 \|f_n\|_{K_{12}}\,,
\quad \|a_n^-\|_{K_{12}}\leq  2 \|f_n\|_{K_{12}}\,.
\end{equation}

Finally, we consider
the remaining case $\max\{\|b_n^-\|_{K_{12}},\|a_n^-\|_{K_{12}}\} = \|a_n^-\|_{K_{12}}
> 4 \|b_n^+\|_{K_{12}}$.
Again by the Schwarz lemma, we obtain
\begin{equation}
\begin{gathered}
\text{On the set where $r\le |z_1|\le R$ and $|z_2|=R$  we have}
\\
\|b_n^-(z)\| \leq \frac {r}{R} \|b_n^-\|_{K_{12}} \leq  \frac {r}{R}  \|a_n^-\|_{K_{12}} \leq  \frac {1}{16}  \|a_n^-\|_{K_{12}}\,.
\end{gathered}
\end{equation}
As $\|a_n^-(z)\|$ achieves its maximum on the set where $|z_1|=r$ and $|z_2|=R$ we then conclude that

\begin{equation}
\|a_n^-\|_{K_{12}}\leq  2 \|f_n\|_{K_{12}},
\end{equation}
which implies
 \begin{equation}
\|b_n^+\|_{K_{12}}\leq  2 \|f_n\|_{K_{12}}, \quad \|b_n^-\|_{K_{12}}\leq  2 \|f_n\|_{K_{12}}\,.
\end{equation}
\end{proof}

\begin{rmk}
For the purposes of the codimension-three conjecture it would have sufficed to prove this theorem for $\dim Y \leq \dim X -3$. Then one can give a slightly different argument which is not substantially different from the argument presented here. We note, however, that in the case $\dim Y \leq \dim X -3$ formal triviality on $W-Z$ amounts to triviality of the associated rank $d$ complex vector bundle.
\end{rmk}

\section{Open problems}

 In this section we discuss open problems which are closely related to our main result.

One can prove the following result.
 \begin{prop}
The category of regular holonomic $\cE_X$-modules is
a full subcategory of the  category of regular holonomic $\hE_X$-modules
\end{prop}

Thus, it is natural to conjecture:

\begin{conj}
The category of regular holonomic $\cE_X$-modules is equivalent
to the  category of regular holonomic $\hE_X$-modules
\end{conj}

Let us fix supports and consider the subcategories or regular holonomic $\cE_X$-modules and regular holonomic $\hE_X$-modules
where the objects have a fixed conic Lagrangian support $\La$. One can check that the two categories
coincide outside a codimension one locus in $\La$ by direct verification.
On the other hand, by making use of Theorem~\ref{9}, it suffices to show that the categories are equivalent outside of a codimension two locus. Thus the problem is reduced to pure codimension one locus on $\La$.
\medskip

Recall that we wrote $\Per_\La$ for the stack of  regular holonomic $\cE_X$-modules
in the introduction where we also discussed the general structure of $\Per_\La$.
In particular, in \cite{GMV2} a description of the stack  $\Per_\La$ is given in terms of the geometry of $\La$ outside of a certain codimension two locus $\La^{\geq 2}$. As our main theorem implies that we can ignore a codimension three locus on $\La$ and we know that for any open $U\subset \ct$ the functor $\Per_\La(U)\to\Per_\La(U\setminus\La^{\geq 2})$
is fully faithful we are left with the following problem:
\begin{problem}
Describe, in terms of the geometry of $\La$ the conditions imposed by the codimension two locus $\La^2$ which cut out  the subcategory $\Per_\La(U)$ in $\Per_\La(U\setminus\La^{\geq 2})$.
\end{problem}

Finally, microlocal perverse sheaves can be defined with arbitrary
coefficients.
Thus we can consider the  stack $\Per_\La(\cors)$ of microlocal perverse sheaves
with coefficients in $\cors$
on $\ct$ supported on $\La$. It is natural to conjecture:

\begin{conj}
The codimension-three conjecture holds for the stack $\Per_\La(\cors)$
when $\cors$ is, for example,  $\bZ$ or a field of positive characteristic.
\end{conj}
 We have neither any evidence nor a strategy of proof for this conjecture.

\end{document}